\documentclass[12pt,a4paper]{article}

\usepackage{amsfonts}
\usepackage{amsmath,bm}
\usepackage{amssymb}
\usepackage{array}
\usepackage{bbold}
\usepackage{blkarray}
\usepackage{citesort,url}
\usepackage{color}
\usepackage{epsfig}
\usepackage{enumitem}
\usepackage[multiple]{footmisc}				
\usepackage{float}
\usepackage{graphicx}
\usepackage[latin1]{inputenc}
\usepackage{multirow}
\usepackage{multicol}
\usepackage{pdfpages}
\usepackage{stmaryrd}							
\usepackage{subfigure}
\usepackage{sidecap}								
\usepackage{xparse}								

\def\sm{\setminus}
\def\wt{\widetilde}

\def\ov{\overline}

\def\({\Bigl (}
\def\){\Bigr )}
\def\[{\Bigl [}
\def\]{\Bigr ]}

\def\dsp{\displaystyle}

\def\x{{\bf x}}
\def\y{{\bf y}}
\def\n{{\bf n}}
\def\q{{\bf q}}

\def\g{{\bf g}}

\def\d{{\rm d}}
\def\aa{{\mathfrak a}}

\def\U{{M_\mu}}
\def\Uj{{M_\mu^j}}
\newcommand{\G}[1][]{ { \Gamma_{#1} } }
\def\O{{\Omega}}
\def\OG{\O\sm\ov\G}

\def\N{\mathbb N}
\def\R{\mathbb R}

\def\D{{\cal D}}
\def\DS{{{\cal D}_S}}

\def\T{{\cal T}}
\def\P{{\cal P}}

\def\dof{\mathit{dof}_{\D}}

\def\trace{\gamma}
\def\grad{{\nabla}}
\def\del{\partial}
\def\div{{\rm div}}
\def\dist{\mbox{\rm dist}}

\DeclareDocumentCommand{\id}{ O{} O{} m }{{#3}_{#1}^{#2}}
\DeclareDocumentCommand{\idl}{ O{l} O{} m }{\id[#2][#1] #3}
\DeclareDocumentCommand{\idt}{ O{n} m }{\id[#1] #2}
\DeclareDocumentCommand{\idtl}{ O{l} O{n} m }{\id[#2][#1] #3}								
\DeclareDocumentCommand{\ida}{ O{\alpha} O{} m }{\id[#2][#1] #3}
\DeclareDocumentCommand{\idal}{ O{l} O{\alpha} O{} m }{\id[#3][#2,#1] #4}
\DeclareDocumentCommand{\idatl}{ O{l} O{n} O{\alpha} m }{\id[#2][#3,#1] #4}
\DeclareDocumentCommand{\idu}{ O{} O{} m }{#3_{#1}^{#2}}
\DeclareDocumentCommand{\idut}{ O{n} m }{\idu[#1] #2}
\DeclareDocumentCommand{\idua}{ O{\alpha} O{} m }{\idu[#2][#1] #3}
\DeclareDocumentCommand{\iduat}{ O{n} O{\alpha} m }{\idu[#1][#2] #3}				
\DeclareDocumentCommand{\idul}{ O{l} O{} m }{\idu[#2][#1] #3}
\DeclareDocumentCommand{\idm}{ O{} O{} m }{#3_{#1}^{#2}}
\DeclareDocumentCommand{\idmt}{ O{n} m }{\idm[#1] #2}
\DeclareDocumentCommand{\idma}{ O{\alpha} O{} m }{\idm[#2][#1] #3}
\DeclareDocumentCommand{\idmat}{ O{n} O{\alpha} m }{\idm[#1][#2] #3}				
\DeclareDocumentCommand{\idf}{ O{} O{} m }{#3_{#1}^{#2}}
\DeclareDocumentCommand{\idft}{ O{n} m }{\idf[#1] #2}
\DeclareDocumentCommand{\idfa}{ O{\alpha} O{} m }{\idf[#2][#1] #3}
\DeclareDocumentCommand{\idfat}{ O{n} O{\alpha} m }{\idf[#1][#2] #3}
\DeclareDocumentCommand{\idSpace}{ O{} m }{ #2^{#1}}
\DeclareDocumentCommand{\idlSpace}{ O{} O{} m }{ #3_{#1}^{#2}}
\DeclareDocumentCommand{\idmSpace}{ O{} m }{ #2_{m}^{#1}}
\DeclareDocumentCommand{\idfSpace}{ O{} m }{ #2_{f}^{#1}}

\DeclareDocumentCommand{\ic}{ O{} O{} m }{\ov #3_{#1}^{#2}}
\DeclareDocumentCommand{\ica}{ O{\alpha} m }{\ic[][#1] #2}		
\DeclareDocumentCommand{\icu}{ O{} O{} m }{\ov #3_{\mu #1}^{#2}}
\DeclareDocumentCommand{\icut}{ O{n} m }{\icu[,#1] #2 }
\DeclareDocumentCommand{\icua}{ O{\alpha} O{} m }{\icu[#2][#1] #3}
\DeclareDocumentCommand{\icuat}{ O{n} O{\alpha} m }{\icu[,#1][#2] #3}

\DeclareDocumentCommand{\iu}{ O{} O{} m }{ #3_{\mu #1}^{#2}}
\DeclareDocumentCommand{\iua}{ O{\alpha} O{} m }{\iu[#2][#1] #3}					
\DeclareDocumentCommand{\iul}{ O{l} O{} m }{\iu[#2][#1] #3}	
\DeclareDocumentCommand{\ia}{ O{\alpha} m }{ #2^{#1}}
\DeclareDocumentCommand{\iaa}{ O{} O{} m }{ #3_{\aa #1}^{#2}}
\DeclareDocumentCommand{\iaaa}{ O{\alpha} m }{\iaa[][#1] #2}
\DeclareDocumentCommand{\ial}{ O{l} O{\alpha} m }{ #3^{#2,#1}}
\DeclareDocumentCommand{\il}{ O{l} m }{ #2^{#1}}
\DeclareDocumentCommand{\icm}{ O{} O{} m }{\ov #3_{m #1}^{#2}}
\DeclareDocumentCommand{\icmt}{ O{n} m }{\icm[,#1] #2 }
\DeclareDocumentCommand{\icma}{ O{\alpha} O{} m }{\icm[#2][#1] #3}
\DeclareDocumentCommand{\icmat}{ O{n} O{\alpha} m }{\icm[,#1][#2] #3}

\DeclareDocumentCommand{\im}{ O{} O{} m }{ #3_{m #1}^{#2}}
\DeclareDocumentCommand{\ima}{ O{\alpha} O{} m }{\im[#2][#1] #3}					
\DeclareDocumentCommand{\imal}{ O{l} O{\alpha} O{} m }{\im[#3][#2,#1] #4}		
\DeclareDocumentCommand{\icf}{ O{} O{} m }{\ov #3_{f #1}^{#2}}
\DeclareDocumentCommand{\icft}{ O{n} m }{\icf[,#1] #2}
\DeclareDocumentCommand{\icfa}{ O{\alpha} O{} m }{\icf[#2][#1] #3}
\DeclareDocumentCommand{\icfat}{ O{n} O{\alpha} m }{\icf[,#1][#2] #3}				

\DeclareDocumentCommand{\iff}{ O{} O{} m }{ #3_{f #1}^{#2}}
\DeclareDocumentCommand{\ifa}{ O{\alpha} O{} m }{\iff[#2][#1] {#3}}	
\DeclareDocumentCommand{\ifl}{ O{l} O{} m }{\iff[#2][#1] {#3}}	
\DeclareDocumentCommand{\ifal}{ O{l} O{\alpha} O{} m }{\iff[#3][#2,#1] {#4}}	
\DeclareDocumentCommand{\icSpace}{ O{} m }{ #2^{#1}}
\DeclareDocumentCommand{\iclSpace}{ O{} O{} m }{ #3_{#1}^{#2}}
\DeclareDocumentCommand{\icmSpace}{ O{} m }{ #2_{m}^{#1}}
\DeclareDocumentCommand{\icfSpace}{ O{} m }{ #2_{f}^{#1}}
\DeclareDocumentCommand{\icuSpace}{ O{} m }{ #2_{\mu}^{#1}}
\DeclareDocumentCommand{\icaaSpace}{ O{} m }{ #2_{\aa}^{#1}}

\newcommand{\COp}{C^\infty_{\Omega}}
\newcommand{\CGp}{C^\infty_{\Gamma}}
\newcommand{\COq}{\mathbf{C}^\infty_{\Omega}}
\newcommand{\CGq}{\mathbf{C}^\infty_{\Gamma}}

\DeclareDocumentCommand{\recu}{ O{} O{} }{\Pi_{\D #1}^{\mu #2}}
\DeclareDocumentCommand{\recgradu}{ O{} O{} }{\grad_{\D #1}^{\mu #2}}
\DeclareDocumentCommand{\Interpolu}{ O{} O{} }{\mathrm I_{\D #1}^{\mu #2}}
\DeclareDocumentCommand{\Interpolul}{ O{,l} O{} }{\Interpolu[#1]}
\DeclareDocumentCommand{\recm}{ O{} O{} }{\Pi_{\D #1}^{m #2}}
\DeclareDocumentCommand{\recmf}{ O{} O{} }{\mathbb T_{\D #1}^{{\aa} #2}}
\DeclareDocumentCommand{\recf}{ O{} O{} }{\Pi_{\D #1}^{f #2}}
\DeclareDocumentCommand{\recfb}{ O{} O{} }{\wt\Pi_{\D #1}^{f #2}}
\DeclareDocumentCommand{\recgradm}{ O{} O{} }{\grad_{\D #1}^{m #2}}
\DeclareDocumentCommand{\recgradf}{ O{} O{} }{\grad_{\D #1}^{f #2}}
\DeclareDocumentCommand{\Interpolm}{ O{} O{} }{\mathrm I_{\D #1}^{m #2}}
\DeclareDocumentCommand{\Interpolml}{ O{,l} O{} }{\Interpolm[#1]}
\DeclareDocumentCommand{\Interpolf}{ O{} O{} }{\mathrm I_{\D #1}^{f #2}}
\DeclareDocumentCommand{\Interpolfl}{ O{,l} O{} }{\Interpolf[#1]}
\DeclareDocumentCommand{\dtd}{ O{} }{\delta_{t}^{#1}}
\DeclareDocumentCommand{\recjump}{ O{} m }{ \llbracket {#2}\rrbracket_{\aa,\D #1} }
\DeclareDocumentCommand{\proj}{ O{} O{} }{\mathrm P_{\D #1}^{#2}}
\DeclareDocumentCommand{\projm}{ O{} O{} }{\mathrm P_{\D #1}^{m #2}}
\DeclareDocumentCommand{\projf}{ O{} O{} }{\mathrm P_{\D #1}^{f #2}}
\DeclareDocumentCommand{\proju}{ O{} O{} }{\mathrm P_{\D #1}^{\mu #2}}

\DeclareDocumentCommand{\dtc}{ O{} }{\del_{t}^{#1}}

\newcommand{\tracemf}{\trace_{\aa}}

\DeclareDocumentCommand{\jump}{ m }{ \llbracket {#1}\rrbracket_\aa }

\DeclareDocumentCommand{\kSu}{ O{} O{} }{ {[kS]}_{\mu #1}^{#2} }
\DeclareDocumentCommand{\kSua}{ O{\alpha} O{} }{\kSu[#2][#1]}
\DeclareDocumentCommand{\kSm}{ O{} O{} }{ {[kS]}_{m #1}^{#2}}
\DeclareDocumentCommand{\kSma}{ O{\alpha} O{} }{\kSm[#2][#1]}
\DeclareDocumentCommand{\kSmf}{ O{} O{} }{ {[kS]}_{mf #1}^{#2}}
\DeclareDocumentCommand{\kSmfa}{ O{\alpha} O{} }{\kSmf[#2][#1]}
\DeclareDocumentCommand{\kSf}{ O{} O{} }{ {[kS]}_{f #1}^{#2} }
\DeclareDocumentCommand{\kSfa}{ O{\alpha} O{} }{\kSf[#2][#1]}
\DeclareDocumentCommand{\kSi}{ O{} O{} }{ {[kS]}_{\aa #1}^{#2} }
\DeclareDocumentCommand{\kSia}{ O{\alpha} O{} }{\kSi[#2][#1]}

\newcommand{\intU}{ \int_{\U} }
\newcommand{\intUT}{ \int_{0}^{T}\int_{\U} }
\newcommand{\LU}{ {L^2(\U)} }
\newcommand{\LUT}{ {L^2((0,T)\times \U)} }
\newcommand{\intO}{ \int_{\O} }
\newcommand{\intG}[1][]{ \int_{\G_{#1}} }
\newcommand{\intT}{ \int_{0}^{T} }
\newcommand{\intGT}[1][]{ \int_{0}^{T}\int_{\G_{#1}} }

\newcommand{\LO}{ {L^2(\O)} }
\newcommand{\LOT}{ {L^2((0,T)\times\O)} }
\newcommand{\LG}[1][]{ {L^2(\G_{#1})} }
\newcommand{\LGT}[1][]{ {L^2((0,T)\times\G_{#1})} }

\newcommand{\du}{\d\tau_\mu}
\newcommand{\dut}{\du\d t}
\newcommand{\dg}{\d\tau_f}

\newcommand{\dtau}{\d\tau}
\newcommand{\dtaut}{\dtau\d t}

\newcommand{\mumf}{{\mu\in\{m,f\}}}
\newcommand{\mumfaa}{{\mumf\cup\chi}}
\newcommand{\aachi}{ {\aa\in\chi} }
\newcommand{\aaa}{ {\aa,\alpha} }

\newcommand{\xim}{{\im{\bm\xi}}}
\newcommand{\xif}{\iff{\bm\xi}}
\newcommand{\xii}{\ifl[i]{\bm\xi}}
\DeclareDocumentCommand{\Tt}{ O{h} }{\mathrm T_{#1}}
\DeclareDocumentCommand{\Tx}{ O{} O{\bm\xi} }{\mathrm T_{#2#1}}
\DeclareDocumentCommand{\Txm}{ O{\xim} }{\mathrm T_{#1}}
\DeclareDocumentCommand{\Txf}{ O{\xif} }{\mathrm T_{#1}}
\DeclareDocumentCommand{\Txi}{ O{\xii} }{\mathrm T_{#1}}

\DeclareDocumentCommand{\FDaaa}{ O{\alpha} O{\aa} O{\D} }{\mathrm F_{#3}^{#2,#1}}
\DeclareDocumentCommand{\FDaaal}{ O{l} O{\alpha} O{\aa} O{\D} }{\mathrm F_{#4^#1}^{#3,#2}}
\DeclareDocumentCommand{\Faaa}{ O{\alpha} O{\aa} O{} }{\mathrm F_{#3}^{#2,#1}}

\newcommand{\eqskip}{\hspace{0.1\textwidth}}

\newcommand{\qed}{\nobreak \ifvmode \relax \else
      \ifdim\lastskip<1.5em \hskip-\lastskip
      \hskip1.5em plus0em minus0.5em \fi \nobreak
      \vrule height0.75em width0.5em depth0.25em\fi}

\newtheorem{theorem}{Theorem}[section]
\newtheorem{remark}[theorem]{Remark}
\newtheorem{definition}[theorem]{Definition}
\newtheorem{lemma}[theorem]{Lemma}

\newtheorem{corollary}[theorem]{Corollary}

\newenvironment{proof}[1][Proof]{\begin{trivlist}
\item[\hskip \labelsep {\bfseries #1}]}{\end{trivlist}}

\usepackage[normalem]{ulem}
\definecolor{violet}{rgb}{0.580,0.,0.827}

\usepackage[normalem]{ulem}
\begin{document}

\title{Numerical analysis of a two-phase flow discrete fracture model}

\author{
J\'er\^ome Droniou\footnote{
School of Mathematical Sciences,
Monash University, Victoria 3800, Australia. \texttt{jerome.droniou@monash.edu}.
}~,
Julian Hennicker{\footnote{
Laboratoire de Math\'ematiques J.A. Dieudonn\'e, 
UMR 7351 CNRS, University Nice Sophia Antipolis, and team COFFEE, INRIA Sophia Antipolis
M\'editerran\'ee, Parc Valrose 06108 Nice Cedex 02, France.
\texttt{julian.hennicker@unice.fr}, \texttt{roland.masson@unice.fr}.
}
\textsuperscript{,}\footnote{
Total SA, Centre scientifique et technique Jean-F\'eger, Avenue Larribau, 64018 Pau, France.
}}~,
Roland Masson{\footnotemark[2]}
}

\maketitle

\begin{abstract}
We present a new model for two phase Darcy flows in fractured media, 
in which fractures are modelled as submanifolds of codimension one with respect to the 
surrounding domain (matrix).
Fractures can act as drains or as barriers, since pressure discontinuities 
at the matrix-fracture interfaces are permitted.
Additionally, a layer of damaged rock at the matrix-fracture interfaces is accounted for.
The numerical analysis is carried out in the general framework of the
Gradient Discretisation Method. Compactness techniques are used to establish convergence results for a wide range of possible numerical schemes;
the existence of a solution for the two phase flow model is obtained as a byproduct
of the convergence analysis.
A series of numerical experiments conclude the paper, with a study of the influence 
of the damaged layer on the numerical solution.
\end{abstract}

\section{Introduction}
\label{sec:intro}
Flow and transport in fractured porous media are of paramount importance for many applications such as 
petroleum exploration and production, geological storage of carbon dioxide, hydrogeology, or geothermal energy. 
Two classes of models, dual continuum and discrete fracture models, 
are typically employed and possibly coupled to simulate flow and transport in fractured porous media. 
Dual continuum models assume that the fracture network is well connected and can 
be homogenised as a continuum coupled 
to the matrix continuum using transfer functions. On the other hand, 
discrete fracture models (DFM), on which this paper focuses, represent explicitly the fractures as co-dimension one 
surfaces immersed in the surrounding matrix domain. 
The use of lower dimensional rather than equi-dimensional entities to represent the fractures has been introduced in 
\cite{MAE02,FNFM03,BMTA03,KDA04,MJE05} 
to facilitate the grid generation and to reduce the number of degrees of freedom of the discretised model. 
The reduction of dimension in the fracture network is obtained from the equi-dimensional model by 
integration and averaging along the width of each fracture. 
The resulting so called hybrid-dimensional model couple the 3D model in the matrix with a 2D model in the fracture network 
taking into account the jump of the normal fluxes as well as additional transmission conditions at the matrix-fracture interfaces. 
These transmission conditions 
depend on the mathematical nature of the equi-dimensional model and on additional physical assumptions.  
They are typically derived for a single phase Darcy flow for which they 
specify either the continuity of the 
pressure in the case of fractures acting as drains \cite{MAE02,GSDFN} or Robin type 
conditions in order to take into account the discontinuity of the pressure
for fractures acting either as drains or barriers \cite{FNFM03,MJE05,ABH09,BHMS16}.

Fewer works deal with the extension of hybrid-dimensional models to two-phase Darcy flows.
Most of them build directly the model at the discrete level as
in \cite{BMTA03,RJBH06,HF08} or are limited to the case of continuous pressures at the matrix-fracture interfaces as in \cite{BMTA03,RJBH06,BGGM14}. 
In \cite{Jaffre11}, an hybrid-dimensional two-phase flow model with discontinuous pressures at the
matrix-fracture interfaces is proposed using a global pressure formulation. However, the transmission conditions at the interface do not take into account correctly the transport from the matrix to the fracture. 

In this paper, a new hybrid-dimensional two-phase Darcy flow model is proposed accounting for complex networks of fractures acting either as drains or barriers. The model takes into account discontinuous capillary pressure curves at the matrix-fracture interfaces. It also includes a layer of damaged
rock at the matrix-fracture interface with its own mobility and capillary pressure functions. This additional layer is not only a modelling tool, it also plays a major role in the numerical analysis of the model
and in the convergence of the non-linear Newton iterations required to solve the discrete equations. 

The discretisation of hybrid-dimensional Darcy flow models 
has been the object of many works using cell-centred Finite Volume schemes with either
Two Point or Multi Point Flux Approximations (TPFA and MPFA) \cite{KDA04,ABH09,HADEH09,TFGCH12,SBN12,AELHP152D,AELHP153D}, 
Mixed or Mixed Hybrid Finite Element methods (MFE and MHFE) \cite{MAE02,MJE05,HF08},
Hybrid Mimetic Mixed Methods (HMM, which contains mixed-hybrid finite volume and mimetic finite difference schemes 
\cite{DEGH09}) \cite{FFJR16,AFSVV16,GSDFN,BHMS16}, 
Control Volume Finite Element Methods (CVFE) \cite{BMTA03,RJBH06,MF07,HADEH09,MMB2007}
or also the Vertex Approximate Gradient (VAG) scheme
\cite{BGGM14,GSDFN,BHMS16,tracer2016}. 
Let us also mention that non-matching discretisations of the fracture and matrix meshes
are studied for single phase Darcy flows in \cite{DS12,FS13,BPS14,SFHW15}.
The convergence analysis for single-phase flow models with a single fracture is established 
in \cite{MAE02,MJE05} for MFE methods, in \cite{DS12} for non matching MFE discretisations,
and in \cite{ABH09} for TPFA discretisations. The case of single-phase flows with complex fracture networks
is studied in the general framework of gradient discretisations in \cite{GSDFN} for continuous pressure models and in \cite{BHMS16} for discontinuous pressure models. For hybrid-dimensional two-phase flow models, the only convergence analysis is to our knowledge done in \cite{BGGM14} for the VAG discretisation of the continuous pressure model with fractures acting only as drains. 
Let us recall that the gradient discretisation method (GDM) enables convergence analysis of both conforming and non conforming discretisations for linear and non-linear second order diffusion and parabolic problems. It accounts for various discretisations such as conforming Finite Element methods, MFE and MHFE methods, some TPFA and symmetric MPFA schemes, and the VAG and HHM schemes \cite{DEH15}.
The main advantage of this framework is to provide the convergence proof 
for all schemes satisfying some abstract conditions at the price of a single convergence analysis for a given model, see e.g. \cite{Eymard.Herbin.ea:2010,EGHM13,DE15,DEGH13,DEF14}.
We refer to the monograph \cite{gdm} for a detailed presentation of the GDM.

The main purpose of this paper is to propose an extension of the gradient discretisation method to our hybrid-dimensional two-phase Darcy flow model. This provides, in an abstract framework, 
the convergence of the solution of the gradient scheme to a weak solution of the model; as
a by-product, this proves the existence of such a solution. 
The numerical analysis is partially based on the previous work \cite{EGHM13}
dealing with the gradient discretisation of single medium two-phase Darcy flows. 
The main new difficulty addressed in this work compared with the analysis of \cite{EGHM13} and \cite{BGGM14}
comes from the transmission conditions at the matrix-fracture interfaces which involve an upwinding between
the fracture phase pressures and the traces of the matrix phase pressures. Note that, as in \cite{EGHM13} and \cite{BGGM14}, the convergence analysis assumes 
that the phase mobilities do not vanish.

The outline of this paper is as follows. Section \ref{sec:model} introduces the geometry of the fracture network, the function spaces, the strong and weak formulations of the model as
well as the assumptions on the data. Section \ref{sec:GS} details the gradient scheme framework including the definition of the abstract reconstruction operators, of the discrete variational formulation, and of the coercivity, consistency, limit conformity and compactness properties. Section \ref{sec:conv} proves the main result of this paper which is the convergence of the gradient scheme solution to a weak solution of the model. This convergence is established through compactness arguments, and requires
to establish various compactness results on the approximation solutions:
averaged in time and space, uniform-in-time and weak-in-space, etc.
The Minty monotonicity trick is used to identify the limit of the non-linear term resulting from the
the upwinding between the fracture and matrix phase pressures. 
Section \ref{sec:num} studies on a 2D numerical example the influence of the additional layer of damaged rock at the matrix-fracture interface on the solution of the model.
The discretisation used in this test case is based on the VAG scheme which can be shown from \cite{BHMS16} to satisfy the assumptions of our gradient discretisation method. Note that 
numerical comparisons of our model with the equi-dimensional model as well as with the continuous pressure model of \cite{BGGM14} 
can be found in \cite{BHMS16b} without the accumulation term in the interfacial
layer, which plays a minor role in the numerical tests when this layer is thin with respect to the fracture (see Section \ref{sec:num}). 
 It is shown that the discontinuous pressure model analysed in this paper is more accurate than
the continuous pressure model of \cite{BGGM14} even in the case of fracture acting only as drains; this improved accuracy is due to
more accurate transmission conditions at the matrix-fracture interfaces, in particular in the case of gravity dominant flows.

\section{Notation and model}
\label{sec:model}
\subsection{Geometry}

Let $\Omega$ denote a bounded domain of $\R^d$ ($d=2,3$),
polyhedral for $d=3$ and polygonal for $d=2$. 
To fix ideas the dimension will be fixed to $d=3$ when it needs to be specified, 
for instance in the naming of the geometrical objects or for the space discretisation 
in the next section. The adaptations to the case $d=2$ are straightforward. 

Let 
$
\overline \Gamma = \bigcup_{i\in I} \overline \Gamma_i
$ 
and its interior $\Gamma = \overline \Gamma\setminus \partial\overline\Gamma$ 
denote the network of fractures $\Gamma_i\subset \Omega$, $i\in I$. Each $\Gamma_i$ is 
a planar polygonal simply connected open domain included in a plane ${\cal P}_i$ of $\R^d$. 
It is assumed that the angles of $\Gamma_i$ 
are strictly smaller than $2\pi$, and that $\Gamma_i\cap\overline\Gamma_j=\emptyset$ for all $i\neq j$.
For all $i\in I$, 
let us set $\Sigma_i = \partial\Gamma_i$, with $\n_{\Sigma_i}$ as unit vector in $\mathcal P_i$, 
normal to $\Sigma_i$ and outward to $\G_i$. Further $\Sigma_{i,j}= \Sigma_i\cap\Sigma_j$ for
$i\not=j$, 
$\Sigma_{i,0} = \Sigma_i\cap\partial\Omega$, 
$\Sigma_{i,N} = \Sigma_i\setminus(\bigcup_{j\in I\sm\{i\}}\Sigma_{i,j}\cup \Sigma_{i,0})$, 
$\Sigma = \bigcup_{(i,j)\in I\times I, i\neq j} ( \Sigma_{i,j}\setminus\Sigma_{i,0} )$ 
and $\Sigma_0 = \bigcup_{i\in I} \Sigma_{i,0}$. 
It is assumed that $\Sigma_{i,0} = \overline\Gamma_i\cap\partial\Omega$. 

\begin{figure}[h!]
\begin{center}
\includegraphics[height=0.25\textwidth]{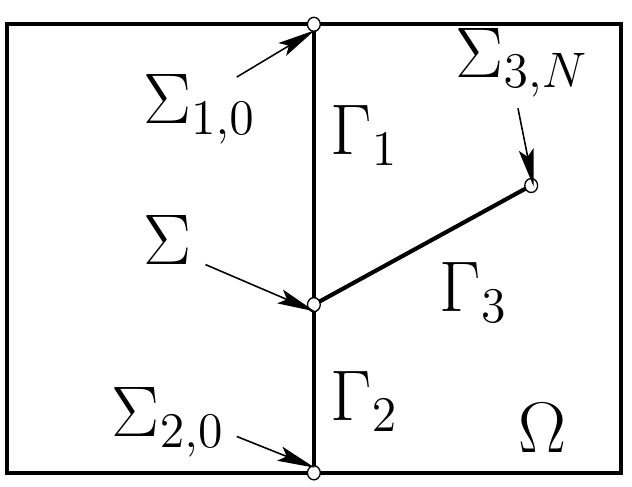}\hspace*{1cm}
\includegraphics[height=0.25\textwidth]{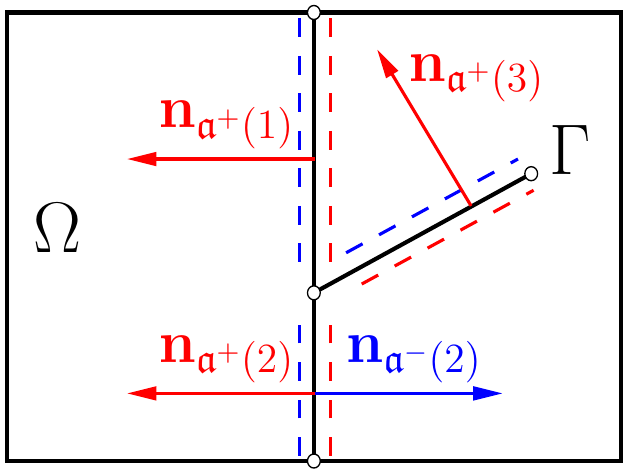}  
\caption{Example of a 2D domain $\Omega$ and 3 intersecting fractures $\Gamma_i,i=1,2,3$. 
We define the fracture plane orientations by $\aa^\pm(i)\in\chi$ for $\Gamma_i$, $i\in I$.}
\label{fig_network}
\end{center}
\end{figure}

We define the two unit normal vectors $\n_{\aa^\pm(i)}$ at each planar fracture $\G_i$,
such that $\n_{\aa^+(i)} + \n_{\aa^-(i)} = 0$ (cf. figure \ref{fig_network}). We define the set of indices
$\chi = \{\aa^+(i), \aa^-(i)\mid i\in I\}$, such that $\#\chi = 2\# I$.
For ease of notation, we use the convention $\G_{\aa^+(i)} = \G_{\aa^-(i)} = \G_i$.

For $\aa = \aa^\pm(i)\in\chi$, we denote by
$
\gamma_{\aa}
$
the one-sided trace operator on $\G_\aa$.
It satisfies the condition $\gamma_{\aa}(h) = \gamma_{\aa}(h\restriction_{\omega_\aa})$, 
where $\omega_\aa = \{\x\in\O\mid(\x-\y)\cdot\n_{\aa}<0,\ \forall\y\in\G_i\}$.

On the fracture 
network $\G$, the tangential gradient is denoted by $\nabla_\tau$, and such that 
$$
\nabla_\tau v = (\nabla_{\tau_i} v_i )_{i\in I},
$$
where, for each $i\in I$, the tangential gradient $\nabla_{\tau_i}$ 
is defined by fixing a 
reference Cartesian coordinate system of the plane ${\cal P}_i$ containing $\G_i$. 
In the same manner, we denote by 
$
\div_{\tau} \q = (\div_{\tau_i} \q_i )_{i\in I}
$
the tangential divergence operator.

\subsection{Continuous model and hypotheses}
\label{subsec:modeleCont}

We describe here the continuous model and assumptions that are implicitly made throughout
the paper.
In the matrix domain $\Omega\setminus\overline\G$, let us denote 
by $\Lambda_m\in L^{\infty}(\Omega)^{d\times d}$ 
the symmetric permeability tensor, chosen such that 
there exist $\overline\lambda_m\geq \underline\lambda_m > 0$ 
with
$$
\underline\lambda_m|\zeta|^2 \leq \Lambda_m(\x)\zeta\cdot\zeta \leq \overline\lambda_m|\zeta|^2 
\mbox{ for all } \zeta \in \R^d, \x\in \Omega.
$$
Analogously, in the fracture network $\G$, we denote by $\Lambda_f\in L^{\infty}(\Gamma)^{(d-1)\times (d-1)}$ the symmetric 
tangential permeability tensor, and assume that there exist $\overline\lambda_f\geq \underline\lambda_f > 0$, such that
$$
\underline\lambda_f|\zeta|^2 \leq \Lambda_f(\x)\zeta\cdot \zeta \leq \overline\lambda_f|\zeta|^2 
\mbox{ for all } \zeta \in \R^{d-1}, \x\in\Gamma. 
$$
On the fracture network $\G$, we introduce an orthonormal system \linebreak
$(\bm\tau_1(\x),\bm\tau_2(\x),\n(\x))$, defined a.e.\ on $\G$. 
Inside the fractures, the normal direction is assumed to be a permeability principal direction. The normal permeability 
$\lambda_{f,\n} \in L^{\infty}(\Gamma)$ is such that 
$\underline \lambda_{f,\n} \leq \lambda_{f,\n}(\x) \leq \overline \lambda_{f,\n}$ for a.e.\ $\x\in \Gamma$ with 
$0 < \underline \lambda_{f,\n} \leq \overline \lambda_{f,\n}$. 
We also denote by $d_f \in L^\infty(\G)$ 
the width of the fractures, assumed to be such that there exist 
${\overline d}_f\geq {\underline d}_f > 0$ with 
$
{\underline d}_f \leq d_f(\x) \leq {\overline d}_f
$
for a.e.\ $\x\in\Gamma$. 
The half normal transmissibility in the fracture network is denoted by 
$$
T_f = \frac{2\lambda_{f,\n}}{d_f}.
$$ 
Furthermore, $\phi_m$ and $\phi_f$ are the matrix and fracture porosities, respectively, $\rho^\alpha\in\R^+$ denotes the density of phase $\alpha$ (with $\alpha=1$ the non-wetting and $\alpha=2$ the wetting phase)
and $\g\in\R^d$ is the gravitational vector field.
We assume that $\underline{\phi}_{m,f}\leq\phi_{m,f}\leq\overline{\phi}_{m,f}$, for some $\underline{\phi}_{m,f},\overline{\phi}_{m,f}> 0$.
$(k_m^\alpha,k_f^\alpha)$ and $(S_m^\alpha,S_f^\alpha)$ are the matrix and fracture phase mobilities and saturations, respectively. Hypothesis on these functions are stated below.

The PDEs model writes: find phase pressures $(\icma u,\icfa u)$ and 
velocities $(\ima \q, \ifa \q)$ ($\alpha = 1,2$), such that
\begin{subequations}
\label{pde:model}
\begin{eqnarray}
\label{modeleCont}
\left\{\begin{array}{r@{\,\,}c@{\,\,}ll}
\dsp \im \phi\del_t \ima S(\icm p) + \div(\ima \q) &=& \ima h &\mbox{ on } (0,T)\times\Omega\setminus \overline \G\\
\dsp \ima \q &=& - \kSma(\icm p) ~\Lambda_m \nabla \icma u 
&\mbox{ on } (0,T)\times\Omega\setminus \overline \G\\ 
\dsp \iff \phi d_f \del_t \ifa S(\icf p) + \div_{\tau}(\ifa \q) - \sum_{\aa\in\chi} \ifa[\alpha][,\aa] Q
&=& d_f \ifa h &\mbox{ on } (0,T)\times\G\\
\dsp \ifa \q &=& -d_f \kSfa(\icf p) ~\Lambda_f \nabla_\tau \icf u &\mbox{ on } (0,T)\times\G \\
\dsp (\icm p,\icf p)\vert_{t=0} &=& (\icmt[0] p,\icft[0] p)&\mbox{ on } (\OG)\times\G.
\end{array}\right.
\end{eqnarray}
The matrix-fracture coupling condition on $(0,T)\times\G_\aa$ (for all $\aa\in\chi$) are
\begin{eqnarray}
\label{CouplingConditionCont}
\left\{\begin{array}{r@{\,\,}c@{\,\,}ll}
\ima \q \cdot \n_\aa + \ifa[\alpha][,\aa] Q &=& \eta\dtc\iaaa S(\tracemf\icm p) & \\
\ifa[\alpha][,\aa] Q &=& \kSfa(\icf p) T_f \jump{\ica{u}}^-		 
- \kSia(\tracemf\icm p) T_f \jump{\ica u } ^+ , 	&					 
\end{array}\right.
\end{eqnarray}
where $\eta = \iaa d\iaa\phi$ with given parameters $\iaa d \in (0,\frac{\iff d}{2})$ and $\iaa \phi \in (0,1]$. In these equations, we have
\begin{align}
\label{ClosureEquations}
\iua[2] S=1-\iua[1] S\mbox{ for $\mumfaa$, and }
(\icm p,\icf p) = (\icma[1] u - \icma[2] u , \icfa[1] u - \icfa[2] u).
\end{align}
\end{subequations}

\begin{SCfigure}[][h]
\label{fig_diffuse_interface_2phase}
\centering
\caption{Illustration of the coupling condition. 
It can be seen as an upwind two point approximation of $\ifa[\alpha][,\aa] Q$.
The upwinding takes into account the damaged rock type at the matrix-fracture interfaces.
The arrows show the positive orientation of the normal fluxes $\ima \q \cdot \n_\aa$ and $\ifa[\alpha][,\aa] Q$.}
\includegraphics[height=0.25\textheight]{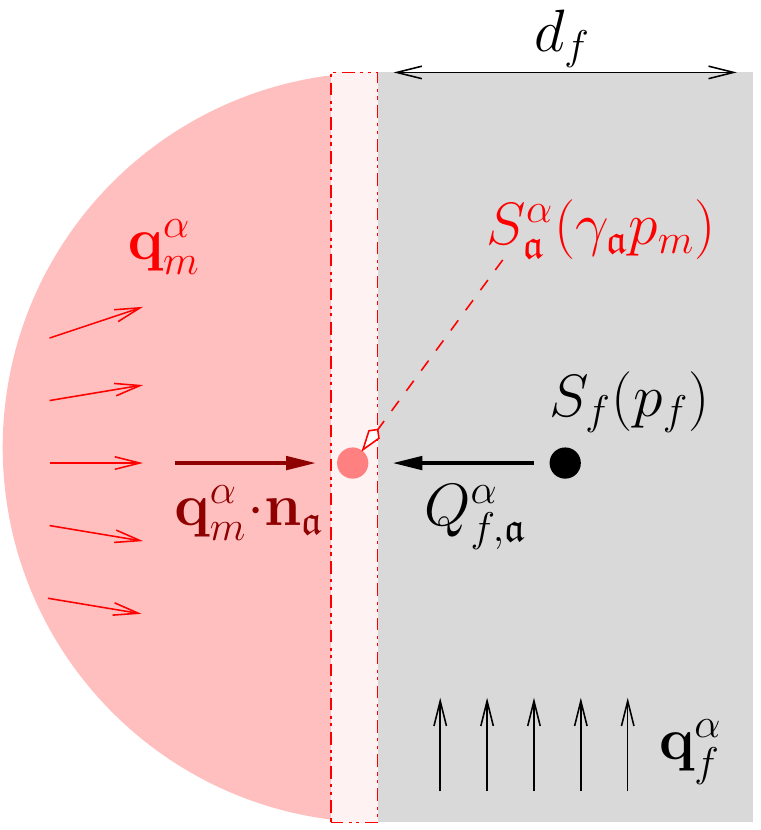} \hspace{2cm}
\end{SCfigure}

In the above, we used the short hand notations 
\begin{align*}
\jump{\ica u} = \tracemf \icma u - \icfa u\,,\quad
\jump{\ica u} ^+ = \max(0,\jump{\ica u})
\quad\text{ and }\quad
\jump{\ica u} ^- = 	\jump{-\ica u}^+														
\end{align*}
as well as, for $\mumfaa$, $\iu \varphi\in\LUT$ and a.e.\ $(t,\x)\in (0,T)\times\U$,
\begin{align*}
\iua S(\iu \varphi)(t,\x) = \iua S(\x,\iu \varphi(t,\x))
\quad\text{ and }\quad
 \kSua (\iu \varphi) (t,\x) = \iua k(\x,\iua S(\x,\iu \varphi(t,\x))).
\end{align*}
Here and in the following, $\U$ is defined by
\begin{align*}
\U = 
\left\{\begin{array}{ll}
\O & \text{ if }\mu = m \\
\G & \text{ if }\mu = f \\
\G[\aa] & \text{ if }\mu = \aa\in\chi.
\end{array}\right.
\end{align*}
The various boundary conditions imposed on the domain are:
homogeneous Dirichlet conditions at the boundary of the domain,
pressure continuity and flux conservation at the fracture-fracture intersections, 
and zero normal flux at the immersed fracture tips. In other words,
\begin{align*}
\trace_{\del\O\sm\del \G} \icm u &= 0 \text{ on }\del\O\sm\del \G,
\qquad
\trace_{\del\O\cap\del\G} \icf u = 0 \text{ on }\del\O\cap\del\G \\
\sum_{i\in I}\q_{f,i}\cdot\n_{\Sigma_i} &= 0\text{ on }\Sigma,
\qquad
\q_{f,i}\cdot\n_{\Sigma_i} = 0\text{ on }\Sigma_{i,N},\ i\in I
\end{align*}

Let us define  $L^2(\Gamma) = \{v = (v_i)_{i\in I}, v_i\in L^2(\Gamma_i), i\in I\}$. 
The assumptions under which the model is considered are:
\begin{itemize}
\item $\icmt[0] p\in H^1(\Omega\setminus\overline\Gamma)$ and $\icft[0] p\in L^2(\G)$,
\item For $\mumf$ and $\alpha=1,2$, $\iua h\in L^2((0,T)\times \U)$,
\item
For $\mumfaa$: $\iua[1] S:\U\times\R\to [0,1]$ is a Caratheodory function; for a.e.\ $\x\in\U$, $\iua[1] S(\x,\cdot)$ 
is a non-decreasing Lipschitz continuous function on $\R$; 
$\iua[1] S(\cdot,q)$ is piecewise constant on a finite partition $(\Uj)_{j\in \iu J}$ of polytopal subsets of $\U$ for all $q\in\R$.
\item For $\alpha=1,2$ and $\mumfaa$: there exist constants $\iu[,{\rm min}] k,\iu[,{\rm max}] k > 0$, such that $\iua k \colon \U\times [0,1] \to [\iu[,{\rm min}] k,\iu[,{\rm max}] k]$ is a Caratheodory function.
\end{itemize}
Recall that a Caratheodory function is measurable w.r.t.\ its first argument and
continuous w.r.t.\ its second argument.

\subsection{Weak formulation}

Let us define the subspace $H^1(\Gamma)$ of $L^2(\Gamma)$ consisting of functions $v = (v_i)_{i\in I}$ such that $v_i\in H^1(\Gamma_i)$ for all $i\in I$, with continuous traces at the fracture intersections $\Sigma_{i,j}$ for all $i\not=j$. 
Its subspace of functions with vanishing traces on $\Sigma_0$ is denoted by $H^1_{\Sigma_0}(\Gamma)$. 

Let us now define the hybrid-dimensional function spaces that 
are used as variational spaces for the Darcy flow model. 
Starting from
$$
V = H^1(\Omega\setminus \overline\G)\times H^1(\G), 
$$
consider the subspace 
$$
V^0 = V_m^0\times V_f^0
$$
where (with $\trace_{\del\O}\colon H^1(\O\backslash\ov\G)\rightarrow L^2(\del\O)$ the trace operator on $\del\O$)
$$
V_m^0 = \{ v\in H^1(\O\backslash\ov\G)\mid \trace_{\del\O} v = 0\text{ on }\del\Omega\}
\quad\mbox{ and }\quad
V_f^0 = H^1_{\Sigma_0}(\G).
$$

The weak formulation of \eqref{pde:model} amounts to finding 
$(\icma u,\icfa u)_{\alpha=1,2}\in [L^2(0,T;\icmSpace[0] V)\times L^2(0,T;\icfSpace[0] V)]^2$ 
satisfying the following variational equalities, for any $\alpha = 1,2$ and any 
$(\icma \varphi,\icfa \varphi) \in C_0^\infty([0, T )\times\O)\times C_0^\infty([0, T )\times\G)$:
\begin{align}
\begin{aligned}
\label{modeleVar}
&
\sum_{\mumf} \(
- \intUT\iu \phi \iua S(\icu p)\del_t\icua\varphi \dut
+ \intUT \kSua(\icu p) ~\iu\Lambda \nabla \icua u \cdot \nabla \icua \varphi \dut \\
&\qquad- 
\intU\iu \phi \iua S(\icut[0] p)\icua \varphi(0,\cdot) \du
\)\\
&+ 
\sum_{\aa\in\chi}\(
\int_0^T\int_{\G_\aa} T_f \( 
\kSia(\tracemf \icm p) \jump{\ica u}^+ 
 - \kSfa(\icf p) \jump{\ica u}^- 
 \) 
 \jump{\ica\varphi} \dtaut \\
&\qquad -
\intGT[\aa] \eta\iaaa S(\tracemf\icm p) \del_t\tracemf\icma\varphi \dtaut
- \intG[\aa] \eta\iaaa S(\tracemf\icmt[0] p) \tracemf\icma\varphi(0,\cdot) \dtau
\) \\
& = 
 \sum_{\mumf}
\intUT \iua h \icua \varphi \du.
\end{aligned}
\end{align}
Here,
\begin{align*}
\du (\x) = 
\left\{\begin{array}{ll}
\d\x & \text{ if }\mu = m \\
\dg(\x) = \iff d(\x)\dtau(\x) & \text{ if }\mu = f
\end{array}\right.
\end{align*}
with $\d\tau({\bf x})$ the $d-1$ dimensional Lebesgue measure on $\Gamma$.

\section{The gradient discretisation method}\label{sec:GS}

The gradient discretisation method consists in selecting a set (gradient discretisation) of
a finite-dimensional space and reconstruction operators on this space,
and in substituting them for their continuous counterpart in the weak formulation
of the model. The scheme thus obtained is called a gradient scheme.
Let us first define the set of discrete elements that make up a gradient discretisation.

\begin{definition}[Gradient discretisation (GD)]\label{def:GD}
A spatial gradient discretisation for a discrete fracture model is
$\D_S = ( \idSpace[0] X, (\recu[_S],\recgradu[_S])_\mumf, (\recjump[_S]{\cdot})_{\aa\in\chi}, (\recmf[_S])_{\aa\in\chi})$, where
\begin{itemize}
\item $\idSpace[0] X$ is a finite dimensional space of degrees of freedom,
\item For $\mumf$,
$\recu[_S] : \idSpace[0] X \rightarrow \LU$ reconstructs a function on $\U$ from
the DOFs,
\item For $\mumf$,
$\recgradu[_S] : \idSpace[0] X \rightarrow \LU^{\dim\U}$ reconstructs a gradient on $\U$
from the DOFs,
\item 
For $\aa\in\chi$, $\recjump[_S]{\cdot} : \idSpace[0] X \rightarrow \LG[\aa]$
reconstructs, from the DOFs, a jump on $\G[\aa]$ between the matrix and fracture,
\item 
For $\aa\in\chi$, $\recmf[_S] : \idSpace[0] X \rightarrow \LG[\aa]$
reconstructs, from the DOFs, a trace on $\G[\aa]$ from the matrix.
\end{itemize}
These operators must be chosen such that the following defines a norm on $\idSpace[0] X$:
\begin{align*}
\| \id w \|_{\D_S} 
&= \(
\|\recgradm[_S] \id w\|_{\LO^{d}}^2
+ \|\recgradf[_S] \idf w\|_{\LG^{d-1}}^2
+ \sum_{\aa\in\chi}\|\recjump[_S]{\id w}\|_{\LG[\aa]} ^2\)^{1/2}.
\end{align*}
The gradient discretisation $\D_S$ is extended to a space-time gradient discretisation 
by setting
$\D=(\D_S,(\Interpolu)_\mumf, (\idt[n] t)_{n=0,\dots,N} )$ with
\begin{itemize}
\item 
$ 0 = \idt[0]t < \idt[1]t < \dots < \idt[N]t = T $ is a discretisation of the time interval $[0,T]$,
\item 
$ \Interpolm \colon H^1(\Omega\setminus\overline\Gamma) \rightarrow \idSpace[0] X $
and $\Interpolf\colon L^2(\G)\rightarrow \idSpace[0]X$ 
are operators designed to interpolate initial conditions.
\end{itemize}
The space-time operators
act on a family $\id u=(\idt[n]u)_{n=0,\ldots,N}\in (\idSpace[0] X)^{N+1}$ the following way:
for all $n = 0,\dots,N-1$ and all $t\in (\idt[n]t,\idt[n+1]t]$,
\begin{align*}
&
\recu \idu u (t,\cdot) = \recu[_S] \idt[n+1] u ,
\qquad 
\recgradu \id u (t,\cdot) = \recgradu[_S] \idt[n+1] u , \\
&
\recmf \id u (t,\cdot) = \recmf[_S] \idt[n+1] u ,
\qquad
\recjump{\id u} (t,\cdot) = \recjump[_S]{\idt[n+1] u} .
\end{align*}
We extend these functions at $t=0$, by considering the corresponding spatial operators on $\idt[0]u$.
\end{definition}

If $w=(\idt[n]w)_{n=0,\ldots,N}$ is a family in $\idSpace[0] X$,
the discrete time derivatives $\dtd w:(0,T]\to \idSpace[0]X$ are defined such that, for all $n = 0,\dots,N-1$
and all $t\in (\idt[n]t,\idt[n+1]t]$, with $\Delta \idt[n+\frac{1}{2}]t = \idt[n+1]t - \idt[n]t$,
\begin{align*}
\dtd w (t) = \frac{\idt[n+1] w - \idt[n] w}{\Delta \idt[n+\frac{1}{2}]t}
\in \idSpace[0] X.
\end{align*}

Let $(e_\nu)_{\nu\in\dof}$ be a basis of $\idSpace[0]X$. If $w\in \idSpace[0]X$, we write
$w=\sum_{\nu\in\dof}w_\nu e_\nu$. Then, for $g\in C(\R)$, we define $g(w)\in \idSpace[0]X$
by $g(w)=\sum_{\nu\in\dof}g(w_\nu)e_\nu$. In other words, $g(w)$ is defined by applying $g$ to
each degree of freedom of $w$. Although this definition depends on the choice of basis $(e_\nu)_{\nu\in\dof}$,
we do not indicate that explicitly. This definition of $g(w)$ is particularly meaningful in the
context of piecewise constant reconstructions, see Remark \ref{rem:pw.rec} below.

The gradient scheme for \eqref{pde:model} consists in writing
the weak formulation \eqref{modeleVar} with continuous spaces and operators substituted
by their discrete counterparts, after a formal integration-by-parts in time. In other words, the gradient scheme is:
find $(\ida u)_{\alpha=1,2} \in [(\idSpace[0] X)^{N+1}]^2$
such that, with $\idu p = \idua[1] u - \idua[2] u $,
\begin{equation}\label{GradDisc:IC}
\idt[0]p= (\Interpolm\icmt[0]p,\Interpolf\icft[0]p)
\end{equation}
and, for any $\alpha = 1,2$ and $\ida v \in (\idSpace[0] X)^{N+1}$,
\begin{align}
&
\sum_{\mumf} \(
\intUT \iu\phi \recu\[\dtd\iua S(\idu p)\] \recu \idua v \dut
+ \intUT \kSua(\recu\idu p) ~\iu\Lambda \recgradu \idua u \cdot \recgradu \idua v \dut 
\)\nonumber\\
&+ 
 \sum_{\aa\in\chi}
\( 
 \int_0^T \int_{\G_\aa} 
\(
\kSia(\recmf\idm p) T_f \recjump{\ida u} ^+
- \kSfa(\recf\idf p) T_f \recjump{\ida u} ^-							
\) 
\recjump{\ida v} \dtaut \nonumber\\
& \qquad+
\intGT[\aa] \eta\recmf\[\dtd\iaaa S(\idm p)\] \recmf\idma v \dtaut
\) 
= 
\sum_{\mumf} \intUT \iua h \recu \idua v \dut.
\label{modeleGradDisc}
\end{align}

\subsection{Properties of gradient discretisations}

The convergence analysis of the GDM is based a few properties that sequences
of GDs must satisfy.

\begin{definition}[Piecewise constant reconstruction operator]
Let $(e_\nu)_{\nu\in\dof}$ be the basis of $\idSpace[0]X$ chosen in Section \ref{sec:GS}.
For $\mumfaa$, an operator $\Pi : \idSpace[0] X \rightarrow \LU$ is called piecewise constant if it has the representation
\begin{align*}
\Pi \idu u = \sum_{\nu\in\dof} u_\nu \mathbb 1_{\omega_\nu^\mu}
\qquad\text{ for all }\idu u=\sum_{\nu\in\dof}u_\nu e_\nu\in\idSpace[0] X,
\end{align*}
where $(\omega_\nu^\mu)_{\nu\in\dof}$ is a partition of $\U$ up to a set of zero measure,
and $\mathbb 1_{\omega_\nu^\mu}$ is the characteristic function of $\omega_\nu^\mu$.
\end{definition}

In the following, all considered function reconstruction operators are assumed to be of piecewise constant type.

\begin{remark}\label{rem:pw.rec}
Recall that, if $g\in C^0(\R)$ and $\idu u\in \idSpace[0]X$, then $g(\idu u) \in \idSpace[0]X$ is
defined by the degrees of freedom $(g(u_\nu))_{\nu\in\dof}$.
Then, any piecewise constant reconstruction operator $\Pi$ commutes with $g$
in the sense that $g(\Pi\idu u) = \Pi g(\idu u)$.
\end{remark}

The coercivity property enables us to control the functions and trace reconstruction
by the norm on $\idSpace[0]X$. This is a combination of a discrete Poincar\'e and a discrete
trace inequality. 

\begin{definition}[Coercivity of spatial GD]
Let
$$
\mathcal C_{\D_S} = 
\max_{0\neq\id v\in \idSpace[0]X}
\frac{
\|\recm[_S]\id v\|_\LO + \|\recf[_S]\id v\|_{\LG} + \sum_{\aa\in\chi} \|\recmf[_S]\id v\|_{\LG[\aa]}
}{
\| \id v\|_{\D_S}
}. 
$$
A sequence $(\D_S^l)_{l\in\N}$ of gradient discretisations is coercive if there exists $\mathcal{C}_P > 0$ such that
$\mathcal C_{\D_S^l}\leq \mathcal C_P$ for all $l\in\N$.
\end{definition}

The consistency ensures that a certain interpolation error goes to zero along sequences of GDs. 

\begin{definition}[Consistency of spatial GD] 
For $\ic u = (\icm u,\icf u)\in \icSpace[0] V$ and $\id v \in \idSpace[0] X$, define 
\begin{align*}
s_{\D_S}(\id v,\ic u) 
={}&
\|\recgradm[_S] \id v -\nabla \icm u\|_{\LO^d} + \|\recgradf[_S] \id v -\nabla_\tau \icf u\|_{\LG^{d-1}} \\
&+ 
\|\recm[_S] \id v - \icm u\|_{\LO} + \|\recf[_S] \id v - \icf u\|_{\LG} \\
&+ 
\sum_{\aa\in\chi}
\(
\|\recjump[_S]{\id v} - \jump{\ic u}\|_{\LG[a]} 
+ \|\recmf[_S] \id v -\tracemf \icm u\|_{\LG[a]}
\),
\end{align*}
and ${\cal S}_{\D_S}(\ic u) = \min_{\id v\in \idSpace[0] X} s_{\D_S}(\id v, \ic u)$. 
A sequence $(\D_S^l)_{l\in\N}$ of gradient discretisations is GD-consistent (or consistent
for short) if, for all $\ic u = (\icm u,\icf u)\in \icSpace[0] V$,
$$
\lim_{l \rightarrow \infty} {\cal S}_{\D_S^l}(\ic u) = 0. 
$$
\end{definition}

To define the notion of limit-conformity, we need the following two spaces:
\begin{align*}
\COq ={}& {C_b^\infty(\OG)}^d\,,\\
\CGq ={}& \Big\lbrace \q_f = (\q_{f,i})_{i\in I}\mid \q_{f,i}\in {C^\infty(\ov\G_i)}^{d-1},\ \sum\nolimits_{i\in I}\q_{f,i}\cdot\n_{\Sigma_i} = 0\text{ on }\Sigma,\\
&\quad \q_{f,i}\cdot\n_{\Sigma_i} = 0\text{ on }\Sigma_{i,N},\ i\in I\Big\rbrace,
\end{align*}
where $C_b^\infty(\OG)\subset C^\infty(\OG)$ is the 
set of functions $\varphi$, such that for all $\x\in\O$ there exists $r > 0$, such that for all 
connected components $\omega$ of $\{\x + \y\in\R^d\mid |\y| < r\}\cap(\OG)$ one has $\varphi\in C^\infty(\ov\omega)$, 
and such that all derivatives of $\varphi$
are bounded.
The limit-conformity imposes that, in the limit, the discrete gradient and
function reconstructions satisfy a natural integration-by-part formula (Stokes' theorem).

\begin{definition}[Limit-conformity of spatial GD]
For all $\q = (\im \q, \iff \q) \in \COq\times\CGq$, $\iaa\varphi \in C_0^\infty(\G[\aa])$
and $\id v \in \idSpace[0] X$, define 
\begin{align*}
w_{\D_S} (\id v,\q,\iaa\varphi) 
={}& 
\intO \(\recgradm[_S] \id v\cdot \im \q + (\recm[_S] \id v) \div\im \q \) \d\x \\
&+ 
\intG \( \recgradf[_S] \id v \cdot \iff\q
+ (\recf[_S] \id v) \div_\tau\iff\q \) \dtau(\x) \\
&- 
\sum_{\aa\in\chi}\intG[\aa]
\im\q\cdot\n_{\aa} \recmf[_S] \id v \dtau(\x) \\
&+ 
\sum_{\aa\in\chi}\intG[\aa]
\iaa\varphi \( 
\recmf[_S] \id v - \recf[_S] \id v - \recjump[_S]{\id v}
\)\dtau(\x)
\end{align*}
and ${\cal W}_{\D_S}(\q,\iaa\varphi) = \max_{0\neq \id v\in \idSpace[0] X}\frac{1}{\|\id v\|_{\D_S}} | w_{\D_S}(\id v,\q,\iaa\varphi)|$.
A sequence $(\D_S^l)_{l\in\N}$ of gradient discretisations is limit-conforming if, for all $\q = (\im\q, \iff\q) \in \COq\times\CGq$ and all $\iaa\varphi \in C_0^\infty(\G[\aa])$,
$$
\lim_{l \rightarrow \infty} {\cal W}_{\D_S^l}(\q,\iaa\varphi) = 0. 
$$
\end{definition}

\begin{remark}[Domain of ${\cal W}_{\D_S}$]
Usually, the measure ${\cal W}_{\D_S}$ of limit-conformity is defined
on spaces in which the Darcy velocities of solutions to the model are expected to be, not smooth spaces
as $\COq\times\CGq$ \cite[Definition 2.6]{gdm}. However, if we do not aim at obtaining error estimates
(which is the case here, given that such estimates would require unrealistic regular
assumptions on the data and the solution), ${\cal W}_{\D_S}$ only needs to be defined and to converge
to $0$ on spaces of smooth functions -- see Lemma \ref{lemmalimitregularity}.
\end{remark}

For any space dependent function $f$, define $\Tx f (\x) = f(\x+\bm\xi)$.
Likewise, for any time dependent function $g$, let $\Tt g (t) = g(t+h)$.
The compactness property ensures a sort of discrete Rellich theorem
(compact embedding of $H^1_0$ into $L^2$). By the Kolmogorov theorem, this compactness
is equivalent to a uniform control of the translates of the functions.

\begin{definition}[Compactness of spatial GD]
For all $\id v \in \idSpace[0] X$ and 
$\bm\xi=(\im{\bm\xi},\iff{\bm\xi})$, 
with $\im{\bm\xi}\in\R^d$ and $\iff{\bm\xi}=(\ifl[i]{\bm\xi})_{i\in I} \in \bigoplus_{i\in I}\tau(\P_i)$, 
where $\tau(\P_i)$ is the (constant) tangent space of $\P_i$, define
\begin{align*}
\tau_{\D_S}(\id v,\bm\xi) = {}&
\|\Txm\recm[_S]\id v - \recm[_S]\id v\|_{L^2(\R^d)}\\
&+ \sum_{i\in I}
\(\|\Txi\recf[_S]\id v - \recf[_S]\id v\|_{L^2(\P_i)}+
\sum_{\aa=\aa^\pm(i)} \|\Txi\recmf[_S]\id v - \recmf[_S]\id v\|_{L^2(\P_i)}
\),
\end{align*}
where all the functions on $\O$ (resp.\ $\G_i$) have been extended to $\R^d$ (resp.\ $\P_i$)
by $0$ outside their initial domain.
Let ${\cal T}_{\D_S}(\bm\xi) = \max_{0\neq \id v\in \idSpace[0] X}\frac{1}{\|\id v\|_{\D_S}} \tau_{\D_S}(\id v,\bm\xi)$. 
A sequence $(\D_S^l)_{l\in\N}$ of gradient discretisations is compact if
$$
\lim_{|\bm\xi| \rightarrow 0} \sup_{l\in\N} {\cal T}_{\D_S^l}(\bm\xi) = 0.
$$
\label{def:comp}
\end{definition}

All these properties for spatial GDs naturally extend to space--time GDs
with, for the consistency, additional requirements on the time steps and on the interpolants of the initial conditions. 

\begin{definition}[Properties of space-time
gradient discretisations]
A sequence of space-time gradient discretisations $(\D^l)_{l\in\N}$ is
\begin{enumerate}
\item 
Coercive if $(\D_S^l)_{l\in\N}$ is coercive.
\item 
Consistent if 
	\begin{itemize}
	\item[(i)]
	$(\D_S^l)_{l\in\N}$ is consistent, 
	\item[(ii)]
	$\Delta \idl[l] t = \max_{n = 0,\dots,N-1} \Delta\idtl[l][n+\frac{1}{2}] t \to 0$ as $l\to\infty$, and 
	\item[(iii)]
	For all $\icm\varphi\in H^1(\Omega\setminus\overline\Gamma)$, 
	$$
	\| \icm\varphi - \recm[_S^l]\Interpolml[^l] \icm\varphi \|_{L^2(\Omega)}
+	\sum_{\aa\in\chi}\| \tracemf\icm\varphi - \recmf[_S^l]\Interpolml[^l] \icm\varphi \|_{\LG[\aa]}
	\longrightarrow 0 \quad\mbox{as $l\to\infty$}
	$$
	 and, for all $\icf\varphi\in L^2(\G)$,
	$
	\| \icf\varphi - \recf[_S^l]\Interpolfl[^l] \icf\varphi \|_{L^2(\G)}\longrightarrow 0
	$ as $l\to\infty$.
	\end{itemize}
\item Limit-conforming if $(\D_S^l)_{l\in\N}$ is limit-conforming.
\item Compact if $(\D_S^l)_{l\in\N}$ is compact.
\end{enumerate}
\end{definition}

Elements of $(\idSpace[0] X)^{N+1}$ are identified with functions
$(0,T]\to \idSpace[0] X$ by setting, for $\id u\in (\idSpace[0] X)^{N+1}$
with $\id u=(\idt[n]u)_{n=0,\ldots,N}$,
\begin{equation}\label{ident:u.fonc}
\forall n=0,\ldots,N-1\,,\;\forall t\in (\idt[n]t,\idt[n+1]t]\,,\;
\id u(t)=\idt[n+1]u.
\end{equation}
This definition is compatible with the choices of space-time operators made
in Definition \ref{def:GD}, in the sense that, for any $t\in (0,T]$,
$\recu \idu u (t,\x) = \recu[_S] (\id u(t))(\x)$ (and similarly for the
other reconstruction operators). With the identification \eqref{ident:u.fonc},
the norm on $(\idSpace[0] X)^{N+1}$ is
\begin{align*}
\| \id u \|_{\D} ^2 
&=\intT \| \id u(t) \|_{\D_S} ^2 \d t.
\end{align*}

\section{Convergence analysis}
\label{sec:conv}

In the rest of this paper, when the phase parameter $\alpha$ is absent it implicitly
mean that it is equal to $1$ so, e.g., we write $\iu S$ for $\iu[][1] S$.
The main convergence result is the following.

\begin{theorem}[Convergence Theorem]

Let $(\D^l)_{l\in \N}$ be a coercive, consistent, limit-confor\-ming and compact sequence of space-time gradient discretisations,
 with piecewise constant reconstructions.
Then for any $l\in\N$ there is a solution $(\id[][\alpha,l] u)_{\alpha=1,2}$ of
\eqref{modeleGradDisc} with $\D=\D_l$.

Moreover, there exists $(\ica u)_{\alpha=1,2} = (\icma u,\icfa u)_{\alpha=1,2}\in [L^2(0,T;\icmSpace[0] V)\times L^2(0,T;\icfSpace[0] V)]^2$
solution of \eqref{modeleVar},
such that, up to a subsequence as $l\to\infty$,
\begin{enumerate}
\item The following weak convergences hold, for $\alpha=1,2$,
\begin{align}
\begin{aligned}
\label{eq:cv.sol}
\left\{
\begin{array}{r@{\,\,}c@{\,\,}l}
&\recu[^l] \idl[\alpha,l]u \rightharpoonup\icua u 
&\quad\mbox{ in } \LUT\,,\mbox{ for $\mumf$},\\
&\recgradu[^l] \idl[\alpha,l]u \rightharpoonup \nabla \icua u 
&\quad\mbox{ in } \LUT^{\mathrm{dim}_\U}\,,\mbox{ for $\mumf$},\\
&\recmf[^l] \idl[\alpha,l]u \rightharpoonup \tracemf \icma u 
&\quad\mbox{ in } \LGT[\aa]\,,\text{ for all }\aa\in\chi,\\
&\recjump[^l]{\idl[\alpha,l]u} \rightharpoonup \jump{\ica u}
&\quad\mbox{ in } \LGT[\aa]\,,\text{ for all }\aa\in\chi.
\end{array}
\right.
\end{aligned}
\end{align}

\item The following strong convergences hold, with $p = \ida[1]u-\ida[2]u$ and $\icu p = \icua[1]u-\icua[2]u$:
\begin{align}
\begin{aligned}
\label{eq:cv.sat}
\left\{
\begin{array}{r@{\,\,}c@{\,\,}l}
&\recu[^l] \iu S(\idl p) \rightarrow \iu S(\icu p)
&\quad\mbox{ in } \LUT\,,\mbox{ for $\mumf$},\\
&\recmf[^l] \iaa S(\idl p) \rightarrow \iaa S(\tracemf\icm p)
&\quad\mbox{ in } \LGT[\aa]\,,\mbox{ for all $\aa\in \chi$}.
\end{array}
\right.
\end{aligned}
\end{align}

\end{enumerate}

\begin{remark}
It is additionally proved in \cite{DHM.fvca8} that the
saturations converge uniformly-in-time strongly in $L^2$ (that is,
in $L^\infty(0,T;L^2)$).
\end{remark}


\label{th:main.cv}
\end{theorem}

\subsection{Preliminary estimates}

Let us introduce some useful auxiliary functions.
These functions are the same as in \cite{DE15,DET16}, with basic adjustment to
account for the fact that the saturation might not vanish at $p=0$.
For $\mumfaa$, let $R_{S_\mu(\x,\cdot)}$ be the range of $S_\mu(\x,\cdot)$.
The pseudo-inverse of $S_\mu(\x,\cdot)$ is the mapping
$[S_\mu(\x,\cdot)]^i:R_{S_\mu(\x,\cdot)}\to\R$ defined by
$$
[S_\mu(\x,\cdot)]^i(q)=\left\{
\begin{array}{ll}
\inf\{z\in\R\,|\,S_\mu(\x,z)=q\}&\mbox{ if $q>S_\mu(\x,0)$}\,,\\
0&\mbox{ if $q=S_\mu(\x,0)$}\,,\\
\sup\{z\in\R\,|\,S_\mu(\x,z)=q\}&\mbox{ if $q<S_\mu(\x,0)$}.
\end{array}\right.
$$
That is, $[S_\mu(\x,\cdot)]^i(q)$ is the point $z$ in $R_{S_\mu(\x,\cdot)}$
that is the closest to $S_\mu(\x,0)$ and such that $S_\mu(\x,z)=q$.
The extended function $B_\mu(\x,\cdot):\R\to [0,\infty]$ is given by
$$
B_\mu(\x,q)=\left\{
\begin{array}{ll}
\dsp \int_{S_\mu(\x,0)}^q [S_\mu(\x,\cdot)]^i(\tau)\d\tau &\mbox{ if } q\in R_{S_\mu(\x,\cdot)}\,, \\
\dsp \infty &\mbox{ else.}\\
\end{array}\right.
$$
$B_\mu(\x,\cdot)$ is convex lower semi-continuous (l.s.c.) and satisfies the following properties \cite{DET16}
\begin{equation}\label{propB:1}
B_\mu(\x,S_\mu(\x,r))=\int_0^r \tau \frac{\partial S_\mu}{\partial q}(\x,\tau)\d\tau,
\end{equation}
\begin{equation}
\forall a,b\in \R\,,\; a(S_\mu(\x,b)-S_\mu(\x,a))\le B_\mu(\x,S_\mu(\x,b))-B_\mu(\x,S_\mu(\x,a))
\label{propB:2}
\end{equation}
and, for some $K_0$, $K_1$ and $K_2$ not depending on $\x$ or $r$,
\begin{equation}\label{propB:3}
K_0S_\mu(\x,r)^2-K_1\le B_\mu(\x,S_\mu(\x,r))\le K_2 r^2.
\end{equation}

In the following, we write $A\lesssim B$ for ``$A\le MB$ for a constant
$M$ depending only on an upper bound of $C_\D$ and
on the data in the assumptions of Section \ref{subsec:modeleCont}''.

\begin{lemma}[Energy estimates]
\label{lemmaAprioriEstimateSolution}
Under the assumptions of Section \ref{subsec:modeleCont}, let $\D$ be a gradient discretisation
with piecewise constant reconstructions $\recu$, $\recmf$. Let $(\ida u)_{\alpha=1,2} \in [(\idSpace[0] X)^{N+1}]^2$ 
be a solution of the gradient scheme of \eqref{modeleGradDisc}. Take $T_0\in (0,T]$
and $k\in\{0,\ldots,N-1\}$ such that $T_0\in (t_k,t_{k+1}]$.
Then
\begin{equation}\label{eq:energy.disc}
\begin{aligned}
\sum_\mumf
 \intU \iu\phi &\left[\iu B( \iu S(\recu[_S]\idu p(T_0)))-\iu B( \iu S(\recu[_S]\idt[0] p))\right]\du\\
&+
\sum_{\alpha=1}^2\sum_\mumf\int_0^{T_0}\intU \kSua(\recu\idu p) \iu\Lambda \recgradu\idua u \cdot \recgradu\idua u \dut\\
&+ \sum_{\aa\in\chi}\intG[\aa] \eta \left[\iaa B( \iaa S(\recmf[_S]\idu p(T_0)))-\iaa B( \iaa S(\recmf[_S]\idt[0] p))\right]\dtau\\
&+\sum_{\alpha=1}^2\sum_{\aa\in\chi}\int_0^{T_0}\intG[\aa]\(
\kSia(\recmf\idm p) T_f \recjump{\ida u} ^+
- \kSfa(\recf\idf p) T_f \recjump{\ida u} ^-
\) \recjump{\ida u}\dtaut \\
&\qquad\le \sum_{\alpha=1}^2\sum_{\mumf} \int_0^{t_{k+1}}\intU \iua h \recu \idua u \dut.
\end{aligned}
\end{equation}
As a consequence,
\begin{equation}
\sum_{\alpha=1,2}\| \ida u \|_\D^2
\lesssim 1
+ \sum_\mumf \|\recu\idut[0] p\|_{L^2(\U)}^2
+ \sum_{\aa\in\chi} \|\recmf[_S]\idt[0] p\|_{L^2(\G[\aa])}^2.
\label{main:apriori.est}
\end{equation}
\end{lemma}

\begin{proof}
We remove the spatial coordinate $\x$ in the arguments, when not needed.
Reasoning as in \cite[Lemma 4.1]{DE15}, Property \eqref{propB:2} gives
\begin{align}
\sum_\mumf
\int_0^{t_{k+1}}\intU &\iu\phi \recu\[\dtd \iu S(\id p)\] \recu\id p \dut\nonumber\\
 =&
\sum_\mumf\sum_{n=0}^{k}\intU \iu\phi \left[\iu S(\recu[_S]\idt[n+1] p)- \iu S(\recu[_S]\idt[n] p)\right]
\recu[_S] \idt[n+1] p\du\nonumber\\
 \geq{}&
\sum_\mumf\sum_{n=0}^{k}
 \intU \iu\phi \left[\iu B( \iu S(\recu[_S]\idt[n+1] p))-\iu B( \iu S(\recu[_S]\idt[n] p))\right]\du\nonumber\\
 ={}&
\sum_\mumf
 \intU \iu\phi \left[\iu B( \iu S(\recu[_S]\idu p(T_0)))-\iu B( \iu S(\recu[_S]\idt[0] p))\right]\du
\label{apriori.est:time1}
\end{align}
where we used, by definition, $\recu[_S]\idu p(T_0)=\recu[_S]\idt[k+1] p$.
Similarly, 
\begin{equation}
\int_0^{t_{k+1}}\intG[\aa] \eta \recmf\[\dtd \iaa S(\id p)\] \recmf\id p \dtaut
 \geq
 \intG[\aa] \eta \left[\iaa B( \iaa S(\recmf[_S]\idu p(T_0)))-\iaa B( \iaa S(\recmf[_S]\idt[0] p))\right]\dtau.
\label{apriori.est:time2}
\end{equation}
Equation \eqref{eq:energy.disc} is then obtained by
taking $\ida v = (\id[0][\alpha] u,\ldots,\id[k+1][\alpha] u,0,\ldots,0)$ (for $\alpha=1,2$) in the gradient scheme
\eqref{modeleGradDisc}, by summing
the resulting equations over $\alpha=1,2$, by using \eqref{apriori.est:time1}
and \eqref{apriori.est:time2}, and by reducing the time integrals in the left-hand side
from $[0,t_{k+1}]$ to $[0,T_0]$, due to the non-negativity of the integrands.

The inequality \eqref{main:apriori.est} is the consequence of a few simple estimates on the
terms of \eqref{eq:energy.disc} with $T_0=T$.
For the symmetric diffusion terms (for $\alpha=1,2$ and $\mumf$), we write
\begin{equation}\label{est:diff.terms}
\intUT \kSua(\recu\idu p) \recgradu\idua u \cdot \recgradu\idua u \dut
\geq \iu[,{\rm min}] d \iu[,{\rm min}]k \|\recgradu\idfa u\|_{\LUT}^2.
\end{equation}
The matrix--fracture coupling terms are handled by noticing that,
for any $s\in\R$, $s^+s=(s^+)^2$ and $s^-s=-(s^-)^2$, so that
for $\alpha=1,2$ and $\aa\in\chi$,
\begin{align}
\intT\intG[\aa]&\(
\kSia(\recmf\idm p) T_f \recjump{\ida u} ^+
- \kSfa(\recf\idf p) T_f \recjump{\ida u} ^-
\) 
\recjump{\ida u}\dtaut \nonumber\\
& =
\intT\intG[\aa]\(
\kSia(\recmf\idm p) T_f ( \recjump{\ida u} ^+ )^2
+ \kSfa(\recf\idf p) T_f ( \recjump{\ida u} ^- )^2
\dtaut
\) \nonumber\\
& \gtrsim 
\| \recjump{\ida u} \|_{\LGT[\aa]}^2.
\label{est:coupl.terms}
\end{align}
Here, we used $\kSia(\recmf\idm p)\ge k_{\aa,{\rm min}}$,
$\kSfa(\recf \idf p)\ge \iff[,{\rm min}]k$ and
$|s|^2=(s^+)^2+(s^-)^2$. Using \eqref{propB:3}, \eqref{est:diff.terms} and \eqref{est:coupl.terms} in
\eqref{eq:energy.disc} (with $T_0=T$), Cauchy--Schwarz inequalities lead to
\begin{align*}
\sum_{\alpha=1}^2 &\[
\|\recgradm\idma u\|_{\LOT^d}^2
+ \|\recgradf\idfa u\|_{\LGT^{d-1}}^2
+ \sum_{\aa\in\chi} \|\recjump{\ida u}\|_{\LGT}^2
\]
\\
\lesssim{}&
\sum_\mumf \[
\sum_{\alpha=1}^2 \|\iua h\|_\LUT\|\recu\idua u\|_\LUT+ \|\recu\idut[0] p\|_{L^2(\U)} ^2
\]
+ \|\recmf[_S]\idt[0] p\|_{L^2(\U)} ^2.
\end{align*}
The proof of \eqref{main:apriori.est} is complete by noticing that 
the left-hand side is equal to $\sum_{\alpha=1}^2 \|\idma u\|_\D^2$,
and by using Young's inequality and the definition of $C_\D$ in the right-hand side. \qed
\end{proof}

The existence of a solution to the gradient scheme follows by a standard
fixed point argument based on the Leray--Schauder topological degree,
see e.g. \cite[proof of Lemma 3.2]{BGGM14}
or \cite[Step 1 in the proof of Theorem 3.1]{DEGH13}.

\begin{corollary}
Under the assumptions of Lemma \ref{lemmaAprioriEstimateSolution}, there exists a solution to the gradient scheme \eqref{modeleGradDisc}.
\end{corollary}

We now want to obtain estimates on the discrete time derivatives.
Let the dual norm of $W=[\im w,\iff w,(\iaa w)_{\aa\in\chi}]\in (\idSpace[0]X)^{2+\sharp\chi}$ be defined by
\begin{equation}\label{def:dual.norm}
\begin{aligned}
|W|_{\D_S,*}=\sup\Bigg\{ \sum_{\mumf}\intU &\iu\phi \recu[_S] \iu w \recu[_S] v \du
+ \sum_{\aa\in\chi} \intG[\aa] \eta\recmf[_S] \iaa w \recmf[_S]\idm v \dtau\,:\\
&v\in \idSpace[0]X\,,\; \|v\|_{\D_S}\le 1\Bigg\}
\end{aligned}
\end{equation}

\begin{lemma}[Weak estimate on time derivatives]
\label{lemma:dual.norm.estimate}
Under the assumptions of Section \ref{subsec:modeleCont}, let $\D$ be a gradient discretisation
with piecewise constant reconstructions $\recu$, $\recmf$. Let $(\ida u)_{\alpha=1,2} \in [(\idSpace[0] X)^{N+1}]^2$
be a solution of the gradient scheme of \eqref{modeleGradDisc}. Then,
$$
\int_0^T \left|\[ \dtd\im S(p)(t),\dtd\iff S(p)(t),(\dtd\iaa S(p)(t))_{\aa\in\chi}\]\right|_{\D_S,*}^2\d t
\lesssim 
1 + \sum_{\alpha=1,2}\|\ida u\|_{\D}^2
$$
\end{lemma}

\begin{proof}
Take $\id v\in \idSpace[0]X$ and apply \eqref{modeleGradDisc} with $\alpha=1$ to
the test function $(0,\ldots,0,\id v,0,\ldots,0)$, where $\id v$ is
at an arbitrary position $n$. This shows that,
for all $n=0,\ldots,N$ and $t\in (\idt[n]t,\idt[n+1]t]$
\begin{align*}
&
\sum_{\mumf} \intU \iu\phi \recu\[\dtd\iu S(\idu p)\](t) \recu \id v \du
+ \sum_{\aa\in\chi} \intG[\aa] \eta\recmf\[\dtd\iaa S(\idm p)\](t) \recmf\id v \dtau \nonumber\\
& = 
\sum_{\mumf} 
\(
\intU \[\frac{1}{\Delta \idt[n+\frac{1}{2}]t}\int_{\idt[n]t}^{\idt[n+1]t}\iu h(s)\d s\] \recu \id v \du
- \intU \kSu(\recu\idu p)(t) ~\iu\Lambda \recgradu \id u(t) \cdot \recgradu \id v \du
\)\nonumber\\
&-
 \sum_{\aa\in\chi}  \int_{\G_\aa} 
\(
\kSi(\recmf\idm p)(t) T_f \recjump{\id u(t)} ^+
- \kSf(\recf\idf p)(t) T_f \recjump{\id u(t)} ^-
\) 
\recjump{\id v} \dtau \nonumber\\
&\lesssim
\left\Vert\frac{1}{\Delta \idt[n+\frac{1}{2}]t}\int_{\idt[n]t}^{\idt[n+1]t}\iu h(s)\d s\right\Vert_{L^2(\U)}
\|\id v\|_{\D_S} +\|\id u(t)\|_{\D_S}\|\id v\|_{\D_S},
\end{align*}
where we have used the definition of $C_\D$ in the last step. Taking the supremum over all
$\id v$ such that $\|\id v\|_{\D_S}=1$ shows that
\begin{multline}
\label{proof_dual.norm.estimate1}
\left|\[ \dtd\im S(p)(t),\dtd\iff S(p)(t),(\dtd\iaa S(p)(t))_{\aa\in\chi}\]\right|_{\D_S,*}\\
\lesssim
\frac{1}{\Delta \idt[n+\frac{1}{2}]t}\int_{\idt[n]t}^{\idt[n+1]t} \| \iu h(s)\|_{L^2(\U)}\d s + \|\id u(t)\|_{\D_S}.
\end{multline}
Take the square of this relation, use $(a+b)^2\le 2a^2+2b^2$, and
apply Jensen's inequality to introduce the square inside the time integral.
Multiply then by $\Delta \idt[n+\frac{1}{2}]t$ and
sum over $n$ to conclude. \qed
\end{proof}

\begin{lemma}[Estimate on time translates]
\label{lemmaTimeTranslates}
Under the assumptions of Section \ref{subsec:modeleCont}, let $\D$ be a gradient discretisation
with piecewise constant reconstructions $\recu$, $\recmf$. 
For any $h>0$ and any solution $(\ida u)_{\alpha=1,2} \in [(\idSpace[0] X)^{N+1}]^2$ 
of \eqref{modeleGradDisc},
\begin{equation}
\begin{aligned}
\sum_\mumf \| \iu S(\Tt\recu \id p) - \iu S(\recu \id p) \|_{\LUT}^2
+ {}&\sum_\aachi \| \iaa S(\Tt\recmf \id p) - \iaa S(\recmf \id p) \|_{\LGT[\aa]}^2 \\
 \lesssim (h+\Delta t) \({}&1+\sum_{\alpha=1}^2\|\ida u\|_\D^2\),
\end{aligned}
\label{eq:ttranslate.main}
\end{equation}
where we recall that $\Tt g(s)=g(s+h)$ and
$\Delta t = \max\{\Delta\idt[n+\frac{1}{2}]t \,:\,n = 0,\dots,N-1\}$, and where all
functions of time have been extended by $0$ outside $(0,T)$.
\end{lemma}

\begin{proof}
Let us start by assuming that $h\in (0,T)$, and let us consider integrals
over $(0,T-h)$ (we therefore do not use extensions outside $(0,T)$ yet).
By the Lipschitz continuity and monotonicity of the saturations $\iu S=\iua[1] S$
we have $|\iu S(b)-\iu S(a)|^2\lesssim (\iu S(b)-\iu S(a))(b-a)$).
Thus, setting $n(s) = \min \{ k=1,\dots,N \mid \idt[k] t \geq s \}$ for all $s\in \R$,
\begin{align}
\sum_\mumf \int_0^{T-h}\intU& | \iu S(\Tt\recu \id p) - \iu S(\recu \id p) |^2 \du\d s
+ \sum_\aachi \int_0^{T-h}\intG[\aa] | \iaa S(\Tt\recmf \id p) - \iaa S(\recmf \id p) |^2 \dtau\d s \nonumber\\
 \lesssim{}&
\sum_\mumf \int_0^{T-h}\intU \iu \phi \( \iu S(\Tt\recu \id p) - \iu S(\recu \id p) \)(s) ( \Tt\recu \id p - \recu \id p)(s)\du \d s\nonumber \\
& \qquad+
\sum_\aachi \int_0^{T-h}\intG[\aa] \eta \(\iaa S(\Tt\recmf \id p) - \iaa S(\recmf \id p) \)(s)( \Tt\recmf \id p - \recmf \id p )(s) \dtau \d s
\nonumber\\
 \lesssim{}&
 \int_0^{T-h} \[
\sum_\mumf \intU \int_{\idt[n(s)]t}^{\idt[n(s+h)]t}
\iu\phi\recu\[\dtd \iu S( \id p)\](t) ( \Tt\recu \id p - \recu \id p )(s) \d t\du \nonumber\\
& \qquad+
\sum_\aachi \intG[\aa] \int_{\idt[n(s)]t}^{\idt[n(s+h)]t}
\eta\recmf\[\dtd \iaa S( \id p)\](t) ( \Tt\recmf \id p - \recmf \id p )(s) \d t\dtau
\] \d s.
\label{proof_ttranslates1}
\end{align}
In the last line, we simply wrote $\iu S(\Tt\recu \id p)(s) - \iu S(\recu \id p)(s)=
\iu S(\recu\id p)(s+h)-\iu S(\recu \id p)(s)$ as the sum
of the jumps if $\iu S(\recu \id p)$ between $s$ and $s+h$ (likewise for $\iaa S(\recmf \id p)$).

For a fixed $s$, define $\id v\in (\idSpace[0] X)^{N+1}$ by
\begin{align*}
\idt[k]v = 
\left\{\begin{array}{ll}
\idt[n(s+h)]p - \idt[n(s)]p & \text{ if } n(s)+1\leq k \leq n(s+h) \\
0 & \text{ else. }
\end{array}\right.
\end{align*}
With this choice,
\begin{align}
\begin{aligned}
\label{proof_ttranslates2}
\recu \id v (t,\x) &= 
\mathbb 1_{[\idt[n(s)]t,\idt[n(s+h)]t]}(t) 
\;
( \Tt\recu \id p - \recu \id p )(s,\x),\\
\recmf \id v (t,\x) &= 
\mathbb 1_{[\idt[n(s)]t,\idt[n(s+h)]t]}(t) 
\;
( \Tt\recmf \id p - \recmf \id p )(s,\x),\\
\recgradu \id v (t,\x) &= 
\mathbb 1_{[\idt[n(s)]t,\idt[n(s+h)]t]}(t) 
\;
( \Tt\recgradu \id p - \recgradu \id p )(s,\x)\,,\text{ and }\\
\recjump{\id v} (t,\x) &= 
\mathbb 1_{[\idt[n(s)]t,\idt[n(s+h)]t]}(t) 
\;
( \Tt\recjump{\id p} - \recjump{\id p} )(s,\x).
\end{aligned}
\end{align} 

We keep $s$ fixed and concentrate on the integrand of the outer integral in the right-hand side of \eqref{proof_ttranslates1}.
Estimate \eqref{proof_dual.norm.estimate1}, the definition \eqref{def:dual.norm} of 
$|\cdot|_{\D_S,*}$, and Young's inequality yield
\begin{align*}
\sum_\mumf {}&\intUT\iu\phi
\recu \[\dtd \iu S(\id p)\] \recu\id v \dut 
+\sum_\aachi \intGT[\aa]\eta
\recmf\[\dtd \iaa S( \id p)\] \recmf\id v \dtaut\\
& \lesssim 
\intT ( \|\iu h(t)\|_{L^2(\U)} + \| \id u(t)\|_{\D_S} ) \| v \|_{\D_S} 
\mathbb 1_{[\idt[n(s)]t,\idt[n(s+h)]t]} (t) \d t\\
& \lesssim 
\intT ( \|\iu h(t)\|_{L^2(\U)} + \| \id u(t)\|_{\D_S} )^2
\mathbb 1_{[\idt[n(s)]t,\idt[n(s+h)]t]} (t) \d t
+(\idt[n(s+h)]t-\idt[n(s)]t) \| v \|_{\D_S}^2.
\end{align*}
Returning to \eqref{proof_ttranslates1}, integrate the previous estimate over $s\in(0,T-h)$.
In this step, it is crucial to realise that
$$
\idt[n(s+h)]t - \idt[n(s)]t
\leq h + \Delta t\ \mbox{ and }\ 
\int_0^{T-h}\mathbb 1_{[\idt[n(s)]t,\idt[n(s+h)]t]} (t) \d s
\leq
\intT \mathbb 1_{[ t-h-\Delta t, t]} (s) \d s 
\leq 
h + \Delta t.
$$
Hence, recalling the definition of $v$,
\begin{align*}
\mathrm{RHS}\eqref{proof_ttranslates1}
\lesssim{}&
(h+\Delta t)
\Bigg[
\intT(\|\iu h(t)\|_{L^2(\U)}+\|\id u(t)\|_{\D_S})^2\d t \\
&\eqskip + \int_0^{T-h}\|p_{n(s+h)}\|_{\D_S}^2\d s
+\int_0^{T-h}\|p_{n(s)}\|_{\D_S}^2\d s
\Bigg]\\
\lesssim{}&
(h+\Delta t)
\(
1
+ \| \id u \|_\D ^2 
+ \| p \|_\D ^2 
\).
\end{align*}
Since $p=\ia[1]u-\ia[2]u$, this proves \eqref{eq:ttranslate.main} with $L^2(0,T-h)$ norms in
the left-hand side, instead of $L^2(0,T)$ norms.
The complete form of \eqref{eq:ttranslate.main} follows by recalling
that $0\le \iu S\le 1$, so that
$\| \iu S(\recu \id p)\|_{L^2((T-h,T)\times\U)}^2\le h$
(and similarly for other saturation terms).
\qed
\end{proof}

\begin{lemma}[Estimate on space translates]
\label{lemmaSpaceTranslates}
Under the assumptions of Section \ref{subsec:modeleCont}, let $\D$ be a gradient discretisation
with piecewise constant reconstructions $\recu$, $\recmf$. 
Let $(\ida u)_{\alpha=1,2} \in [(\idSpace[0] X)^{N+1}]^2$ be a solution of \eqref{modeleGradDisc},
and let $\bm\xi=(\im{\bm\xi},\iff{\bm\xi})$, 
with $\im{\bm\xi}\in\R^d$ and $\iff{\bm\xi}=(\ifl[i]{\bm\xi})_{i\in I} \in \bigoplus_{i\in I}\tau(\P_i)$, 
where $\tau(\P_i)$ is the (const.) tangent space of $\P_i$. Then, extending
the functions $\recu \id p$ and $\iu S$ by $0$ outside $\U$,
\begin{align*}
&\| \Txm\im S(\recm \id p) - \im S(\recm \id p) \|_{L^2((0,T)\times\R^d)}^2+ \sum_{i\in I}\(
\| \Txi\iff S(\recf \id p) - \iff S(\recf \id p) \|_{L^2((0,T)\times\P_i)}^2
\\
&\eqskip 
+\sum_{\aa=\aa^\pm(i)} \| \Txi\iaa S(\recmf \id p) - \iaa S(\recmf \id p) \|_{L^2((0,T)\times\P_i)}^2 
\)
\lesssim
{\cal T}_{\D_S}(\bm\xi) \sum_{\alpha=1}^2\|\ida u\|_\D^2
+ |\bm\xi|,
\end{align*}
where we recall that ${\rm T}_{\bm\zeta} f(\x)=f(\x+{\bm\zeta})$,
and ${\cal T}_{\D_S}$ is given in Definition \ref{def:comp}.
\end{lemma}

\begin{proof}
Let us focus on the matrix $\O$, and remember that, as a function of $\x$,
$\im S$ is piecewise constant on a polytopal partition $(\Omega_j)_{j\in \im J}$.
Write
\begin{align}
\Txm\im S(\recm \id p) - \im S(\recm \id p)
={}& \im S(\x+\xim,\recm \id p(\x+\xim,t)) - \im S(\x+\xim,\recm \id p(\x,t))\nonumber\\
&+ \im S(\x+\xim,\recm \id p(\x,t)) - \im S(\x,\recm \id p(\x,t)).
\label{split.Sm}
\end{align}
Let 
$
\O_{\xim} = \bigcup_j\{\x\in\O_j\mid\x+\xim\not\in\O_j\}\cup\{\x\in\R^d\sm\O\mid\x+\xim\in\O\}
$
be the set of points $\x$ that do not belong to the same element $\Omega_j$ as
their translate $\x+\xim$.
By assumption on $\im S$,
\begin{align*}
\sup_{q\in\R} |\im S(\x+\xim,q)-\im S(\x,q)|\le 
\left\{\begin{array}{r@{\,\,}c@{\,\,}l}
&0&\qquad\text{ on }\R^d\sm\O_{\xim},\\
&1&\qquad\text{ on }\O_{\xim}.
\end{array}\right.
\end{align*}
Moreover, since each $\O_j$ is polytopal, $|\Omega_{\xim}|\lesssim |\xim|$.
Hence,
\begin{equation}
\intT\int_{\R^d} \sup_{q\in\R} |\im S(\x+\xim,q)-\im S(\x,q)|^2\d\x\d t \lesssim |\xim|.
\label{tt.Sm.1}\end{equation}
On the other hand, by definition of $\T_{D_S}$ and the Lipschitz continuity
of $\im S$,
\begin{align}
&\intT\int_{\R^d} |\im S(\x+\xim,\recm \id p(\x+\xim,t)) - \im S(\x+\xim,\recm \id p(\x,t))|^2 \d\x\d t \nonumber\\
&\eqskip\lesssim
\intT\int_{\R^d} |\recm \id p(\x+\xim,t) - \recm \id p(\x,t)|^2 \d\x\d t \lesssim \| p \|_\D^2 \T_{D_S}(\bm\xi).
\label{tt.Sm.2}
\end{align}
Plugging \eqref{tt.Sm.1} and \eqref{tt.Sm.2} into \eqref{split.Sm}
and reasoning similarly for $\iff S$ and $\iaa S$ concludes the proof.
\qed
\end{proof}

\begin{remark} This proof is the only place where the assumption that
each $\Uj$ is polytopal is used; this is to ensure that $|\O_\xim|\lesssim |\xim|$
(and likewise for fracture and interfacial terms).
Obviously, this asssumption on the sets $\Uj$ could be relaxed (e.g., into
``each $\Uj$ has a Lipschitz-continuous boundary''), but assuming that these sets
are polytopal is not restrictive for practical applications.
\end{remark}

\subsection{Initial convergences}

We can now state our initial convergence theorem for sequences of solutions
to gradient schemes. This theorem does not yet identify the weak limits of such
sequences.

\begin{theorem}[Averaged-in-time convergence of approximate solutions]~\\
Let $(\D^l)_{l\in \N}$ be a coercive, consistent, limit-conforming and compact sequence of space-time gradient discretisations, with piecewise constant reconstructions.
Let $( \id[][\alpha,l] u )_{\alpha=1,2\,,l\in \N}$ be such that $(\id[][\alpha,l]u)_{\alpha=1,2}\in [(\idlSpace[l][0]X)^{N_l+1}]^2$ is a solution of \eqref{modeleGradDisc} with $\D=\D_l$.
Then, there exists $(\ica u)_{\alpha=1,2} = (\icma u,\icfa u)_{\alpha=1,2}\in [L^2(0,T;\icmSpace[0] V)\times L^2(0,T;\icfSpace[0] V)]^2$
such that, up to a subsequence as $l\to\infty$,
the convergences \eqref{eq:cv.sol} and \eqref{eq:cv.sat} hold.
\label{th:cv}
\end{theorem}

\begin{proof}
Combining Lemmata \ref{lemmaAprioriEstimateSolution} and \ref{lemmalimitregularity}
immediately gives \eqref{eq:cv.sol}.
By assumption, $0 \leq \iu S,\iaa S \leq 1$ and therefore, by Lemmata \ref{lemmaTimeTranslates} and \ref{lemmaSpaceTranslates} and the Kolmogorov compactness theorem,
there exists a subsequence of $(\recu[^l] \iu S(\idl[l]p))_l$ that strongly converges in $\LUT$
and a subsequence of $(\recmf[^l] \iaa S(\idl[l]p))_l$ that strongly converges in $\LGT[\aa]$.
Also, by assumption, $\iu S,\iaa S$ are non-decreasing functions, which allows us to identify the limits in \eqref{eq:cv.sat}
by applying Corollary \ref{corollary:limitregularity}.
\qed
\end{proof}

Let $\COp$ be the subspace of functions in $C_b^\infty(\OG)$ vanishing on a neighbourhood 
of the boundary $\partial\Omega$. Define also $\CGp = \gamma_\Gamma (C_0^\infty(\O))$ as the image of $C_0^\infty(\O)$ 
of the trace operator $\gamma_\Gamma\colon H_0^1(\O)\to L^2(\G)$.

The following lemma and theorem add a uniform-in-time weak $L^2$ convergence
property to the properties already established in Theorem \ref{th:cv}.

\begin{lemma}[Uniform-in-time, weak-in-space translate estimates]
\label{lemma:condition1.Ascoli}
Under the assumptions of Section \ref{subsec:modeleCont}, let $\D$ be a gradient discretisation
with piecewise constant reconstructions $\recu$, $\recmf$. 
Let $(\ida u)_{\alpha=1,2} \in [(\idSpace[0] X)^{N+1}]^2$ 
be a solution of the gradient scheme of \eqref{modeleGradDisc}, and $p=\ida[1]u-\ida[2]u$.
Then, for all $\varphi = ( \im\varphi,\iff\varphi ) \in C_\O^\infty\times C_{\G}^\infty$
and all $s,t\in [0,T]$,
\begin{align}
&\left|\sum_\mumf 
	\left\langle \iu d\iu\phi \recu \iu S(\id p)(s)- \iu d\iu\phi \recu \iu S(\id p)(t),\iu\varphi \right\rangle_\LU 
	 \right.\nonumber\\
&\eqskip+ 
	\left.\sum_\aachi
	\left\langle \eta \recmf \iaa S(\id p)(s)- \eta \recmf \iaa S(\id p)(t),\tracemf\im\varphi \right\rangle_{\LG[\aa]} 
	 \right|\nonumber\\
&\eqskip \lesssim
	\mathcal S_\DS(\varphi)	
	+(\mathcal S_\DS(\varphi)+C_\varphi) \left(1 + \sum_{\alpha=1}^2\|\ida u\|_{\D}^2\right)^{\frac{1}{2}}\[ |s-t|^\frac{1}{2} + (\Delta t)^\frac{1}{2} \].	
\label{weak.tt}
\end{align}
where $C_\varphi$ only depends on $\varphi$, $\iff d$ is the width of the fractures, and $\im d=1$.
\end{lemma}

\begin{proof}
Let us introduce an interpolant $\proj[_S] : C_\O^\infty\times C_{\G}^\infty \rightarrow \idSpace[0] X$ 
such that, for all $\varphi\in C_\O^\infty\times C_{\G}^\infty$,
\begin{align*}
s_\DS(\proj[_S]\varphi,\varphi)
 = 
\mathcal S_\DS(\varphi).
\end{align*}
As in the proof of Lemma \ref{lemmaTimeTranslates}, let $n(r) = \min \{ k=1,\dots,N \mid \idt[k] t \geq r \}$ for all $r\in [0,T]$. Denote by $L$ the left-hand side of \eqref{weak.tt} and
introduce $\recu[_S]\proj[_S]\varphi$ in the first sum and $\recmf[_S]\proj[_S]\varphi$ in the second sum to write
\begin{align}
L \leq{}& 
	\sum_\mumf\(
	\left| \left\langle \iu d\iu\phi \recu \iu S(\id p)(s)-\iu d\iu\phi \recu \iu S(\id p)(t),\iu\varphi - \recu[_S]\proj[_S]\varphi \right\rangle_\LU \right|\)\label{term.1}\\
& + 	
	\sum_\aachi\(
	\left| \left\langle \eta \recmf \iaa S(\id p)(s)-\eta \recmf \iaa S(\id p)(t),\tracemf\im\varphi - \recmf[_S]\proj[_S]\varphi \right\rangle_{\LG[\aa]} \right|	\) \label{term.2}\\
& +	
	\left| \sum_\mumf
	\left\langle \iu d\iu\phi \[ \recu \iu S(\id p)(s) - \recu \iu S(\id p)(t) \],\recu[_S]\proj[_S]\varphi \right\rangle_\LU 
	 \right. \nonumber\\
& \eqskip+ 
	\left. \sum_\aachi
	\left\langle \eta \[ \recmf \iaa S(\id p)(s) - \recmf \iaa S(\id p)(t) \],\recmf[_S]\proj[_S]\varphi \right\rangle_{\LG[\aa]}
	 \right| \nonumber\\
\lesssim{}&
	 \mathcal S_\DS(\varphi)
	+ \left| \sum_\mumf
	\left\langle \iu d\iu\phi \[ \recu \iu S(\id p)(s) - \recu \iu S(\id p)(t) \],\recu[_S]\proj[_S]\varphi \right\rangle_\LU 
	 \right. \nonumber\\
&\eqskip\quad + 
	\left. \sum_\aachi
	\left\langle \eta \[\recmf \iaa S(\id p)(s) - \recmf \iaa S(\id p)(t) \],\recmf[_S]\proj[_S]\varphi \right\rangle_{\LG[\aa]}
	 \right|.
\label{est.L}
\end{align}
Here, the terms \eqref{term.1} and \eqref{term.2} have been estimated by
using $0\le \iu S,\iaa S\le 1$ and the definition of $\proj[_S]\varphi$.
Let $L_1$ be the second addend in \eqref{est.L}.
Assuming that $t<s$, and hence $n(t)\leq n(s)$, write $\recu \iu S(\id p)(s) - \recu \iu S(\id p)(t)$ 
and $\recmf \iaa S(\id p)(s) - \recmf \iaa S(\id p)(t)$ as the sum
of their jumps, and recall the definition \eqref{def:dual.norm} of $|\cdot|_{\D_S,*}$ to obtain
\begin{align*}
L_1\le{}&\Bigg| \sum_{k=n(t)}^{n(s)-1} \Delta \idt[k+\frac{1}{2}]t \(
	\sum_\mumf \left\langle \iu d\iu\phi \recu \dtd \iu S(\id p)(\idt[k]t),\recu[_S]\proj[_S]\varphi \right\rangle_\LU \\
&\qquad\qquad\qquad\qquad+ \sum_{\aa\in\chi} \left\langle \eta \recmf \dtd \iaa S(\id p)(\idt[k]t),\recmf[_S]\proj[_S]\varphi \right\rangle_\LU 
	\)\Bigg| \\
\leq{}&
	\sum_{k=n(t)}^{n(s)-1} \Delta \idt[k+\frac{1}{2}]t
	\left| \[ \dtd\im S(p)(\idt[k]t),\dtd\iff S(p)(\idt[k]t),(\dtd\iaa S(p)(\idt[k]t))_{\aa\in\chi}\]\right|_{\D_S,*}
	\|\proj[_S]\varphi\|_\DS \\
\leq{}&
	\|\proj[_S]\varphi\|_\DS 
	\intT \mathbb{1}_{[ \idt[n(t)]t,\idt[n(s)]t ]}(r)
	\left| \[ \dtd\im S(p)(r),\dtd\iff S(p)(r),(\dtd\iaa S(p)(r))_{\aa\in\chi}\]\right|_{\D_S,*} \d r.
\end{align*}
Use now Lemmata \ref{lemma:dual.norm.estimate} and the Cauchy--Schwarz inequality to infer
\begin{equation}\label{est.L1}
L_1\lesssim 
	\|\proj[_S]\varphi\|_\DS\left(1 + \sum_{\alpha=1}^2\|\ida u\|_{\D}^2\right)^{\frac{1}{2}}\[ (s-t)^\frac{1}{2} + (\Delta t)^\frac{1}{2} \].
\end{equation}
By the triangle inequality,
\begin{align*}
\|\proj[_S]\varphi\|_\DS
\leq
\mathcal S_\DS(\varphi)
+ \|\grad \im\varphi\|_{\LO^{d}}
+ \|\grad_\tau \iff\varphi\|_{\LG^{d-1}}
+ \sum_{\aa\in\chi}\|\jump{\varphi}\|_{\LG[\aa]}=\mathcal S_\DS(\varphi)+C_\varphi.
\end{align*}
Plugging this into \eqref{est.L1} and the resulting inequality into \eqref{est.L} concludes the proof.
\qed

\end{proof}

\begin{theorem}[Uniform-in-time, weak-in-space convergence]
\label{th:cv.unifT-weakL2}
Under the as\-sump\-tions of Theorem \ref{th:cv}, for all $\mumf$ and $\aachi$,
$\iu S(\icu p):[0,T]\to \LU$ and $\iaa S(\tracemf\icm p):[0,T]\to \LG[\aa]$
are continuous for the weak topologies of $\LU$ and $\LG[\aa]$, respectively,
and
\begin{equation}
\begin{aligned}
&\recu[^l] \iu S(\idl p) 
\longrightarrow \iu S(\icu p)
\text{ uniformly in }[0,T]\text{, weakly in }\LU,\\
&\recmf[^l] \iaa S(\idl p) 
\longrightarrow \iaa S(\tracemf\icm p)
\text{ uniformly in }[0,T]\text{, weakly in }\LG[\aa],
\end{aligned}
\end{equation}
where the definition of the uniform-in-time weak $L^2$ convergence is recalled in the appendix.
\end{theorem}

\begin{proof}
The proof hinges on the discontinuous Arzel\`a-Ascoli theorem (Theorem \ref{th:disc.ascoli}
in the appendix).
Consider first the matrix saturation. The space
$\icmSpace {\mathcal R} = \left\lbrace \im d\im\phi\im\varphi \mid \im\varphi\in C_0^\infty(\OG) \right\rbrace$ is dense in $L^2(\O)$.
Apply \eqref{weak.tt} to $\varphi=(\im\varphi,0)$. Since $\iff\varphi=\tracemf\im\varphi=0$
only the term involving $\im S$ remains in the left-hand side.
The resulting estimate and the property $0\le \im S\le 1$
show that the sequence of functions
$(t \mapsto \recm[^l] \im S(\idl p)(t) )_{l\in\N}$
satisfies the assumptions of Theorem \ref{th:disc.ascoli} with $\mathcal R=\icmSpace {\mathcal R}$.
Hence, $(\recm[^l] \im S(\idl p))_{l\in\N}$ has a subsequence that converges
uniformly on $[0,T]$ weakly in $L^2(\O)$. Given \eqref{eq:cv.sat}, the weak limit
of this sequence must be $\im S(\icm p)$.

A similar reasoning, based on the space $\icfSpace {\mathcal R} = \left\lbrace \iff d\iff\phi\iff\varphi \mid \iff\varphi\in C_{\G}^\infty \right\rbrace$ -- which is dense
in $L^2(\Gamma)$ -- and using $\varphi=(0,\iff\varphi)$ in \eqref{weak.tt},
gives the uniform-in-time weak $L^2(\Gamma)$ convergence of 
$\recf[^l] \iff S(\idl p)$ towards $\iff S(\icf p)$.

Let us now turn to the convergence of the trace saturations. Take $\im\varphi\in C_\O^\infty$
such that the support of $\tracemf\im\varphi$ is non empty for exactly one $\aa\in\chi$.
Considering $\varphi=(\im\varphi,0)$ in \eqref{weak.tt} leads to
\begin{align}
&\left|\left\langle \eta \recmf[^l] \iaa S(\idl p)(s)- \eta \recmf[^l] \iaa S(\idl p)(t),\tracemf\im\varphi \right\rangle_{\LG[\aa]} 
	 \right|\nonumber\\
&\eqskip \lesssim
	\mathcal S_\DS(\varphi)	
	+(\mathcal S_\DS(\varphi)+C_\varphi) \left(1 + \sum_{\alpha=1}^2\|\ida u\|_{\D}^2\right)^{\frac{1}{2}}\[ |s-t|^\frac{1}{2} + (\Delta t)^\frac{1}{2} \]\nonumber\\
&\eqskip+\left|\left\langle \im d\im\phi \recm[^l] \im S(\idl p)(s)- \im d\im\phi \recm[^l] \im S(\idl p)(t),\im\varphi \right\rangle_{L^2(\O)} 
	 \right|.
\label{last.term}
\end{align}
Since it was established that $(\im d\im\phi \recm[^l] \im S(\idl p))_{l\in\N}$ converges
uniformly-in-time weakly in $L^2(\O)$, the sequence
$(\langle\im d\im\phi \recm[^l] \im S(\idl p),\im\varphi\rangle_{L^2(\O)})_{l\in\N}$ 
is equi-continuous and the last term in \eqref{last.term} therefore tends to 
$0$ uniformly in $l$ as $s-t\to 0$. Hence, \eqref{last.term} enables the usage
of Theorem \ref{th:disc.ascoli}, by noticing that 
$\{\eta\tracemf\im\varphi\mid \im\varphi\in C_\O^\infty\,,\;
{\rm supp}(\gamma_{\mathfrak b}\im\varphi)=\emptyset\text{ for all }{\mathfrak b}\in\chi\text{ with }{\mathfrak b}\neq\aa
\}$ 
is dense in $\LG[\aa]$,
and gives the uniform-in-time weak $\LG[\aa]$ convergence of $\recmf[^l] \iaa S(\idl p)$. \qed

\end{proof}

\subsection{Proof of Theorem \ref{th:main.cv}}

The proof of the main convergence theorem can now be given.

\paragraph*{First step:} passing to the limit in the gradient scheme.

Let us introduce the family of functions $(\FDaaal)_{\aachi}^{\alpha=1,2}$:
$$
\FDaaal(t,\x,\beta) =
\[ \iff T\kSia(\recmf[^l] \idl p) \beta^+ 
 - \iff T\kSfa(\recf[^l] \idl p) \beta^- \](t,\x),
\quad\text{ for all }\beta\in\LG[\aa],
$$
and their continuous counterparts $(\Faaa)_{\aachi}^{\alpha=1,2}$:
$$
\Faaa(t,\x,\beta) =
\[ \iff T\kSia(\tracemf\icm p) \beta^+ 
 - \iff T\kSfa(\icf p) \beta^- \](t,\x),
\quad\text{ for all }\beta\in\LG[\aa].
$$
The following properties are easy to check.
Firstly, since $\iff T$, $\kSia$ and $\kSf$ are positive 
and $s\mapsto s^+$ and $s\mapsto -s^-$ are non-decreasing,
\begin{equation}
\label{proof.main.1}
\[\FDaaal(t,\x,\beta)-\FDaaal(t,\x,\gamma)\]\[\beta(t,\x) - \gamma(t,\x)\] \ge 0,
\quad\text{ for all }\beta,\gamma\in\LG[\aa].
\end{equation}
Secondly, by the convergences \eqref{eq:cv.sat},
for $(\beta_l)_{l\in\N}\subset\LG[\aa]$ and $\beta\in\LG[\aa]$,
\begin{equation}
\label{proof.main.2}
\beta_l\longrightarrow\beta\ \text{ in }\LGT[\aa]
\quad\Longrightarrow \quad
\FDaaal(\beta_l)\longrightarrow\Faaa(\beta)\ \text{ in }\LGT[\aa].
\end{equation}
Thirdly, by Lemma \ref{lemmaAprioriEstimateSolution}, the sequences $(\FDaaal(\recjump[^l] {\idl u}))_{l\in\N}$ ($\aachi$, $\alpha=1,2$) are bounded in $\LGT[\aa]$ and 
there exists thus $\iaaa\rho\in\LGT[\aa]$
such that, up to a subsequence,
\begin{equation}\label{weak.conv.F}
\FDaaal(\recjump[^l]{\idal u})\rightharpoonup\iaaa\rho\quad\text{ in }\LGT[\aa].
\end{equation}

Consider $\ia\varphi = (\ima \varphi,\ifa \varphi) = \sum_{k=1}^{b} \ial[k]\theta\otimes\ial[k]\psi$,
where $(\ial[k]\psi)_{k\in\N} = (\imal[k] \psi,\ifal[k] \psi)_{k=1,\ldots,b} \in C_\O^\infty\times C_{\G}^\infty$
and 
$(\ial[k]\theta)_{k=1,\ldots,b} \in C_0^\infty([0,T))$.
Take $(\idatl v)_{n=0,\dots,\idl N} = (\proj[_S^l]\ia\varphi(\idtl t))_{n=0,\dots,\idl N}\in(\idlSpace[l][0]X)^{\idl N+1}$ 
as ``test function'' in \eqref{modeleGradDisc}. Here, $\proj[_S^l]$
is defined as in the proof of Lemma \ref{lemma:condition1.Ascoli}.
Apply the discrete integration-by-parts of \cite[Section C.1.6]{gdm} on the accumulation terms in \eqref{modeleGradDisc}, let $l\to\infty$ and use
standard convergence arguments \cite{DE15,gdm} based on Theorem \ref{th:cv} 
to see that
\begin{align}
\begin{aligned}
\label{proof.main.0}
&
\sum_{\alpha=1}^2\Bigg\lbrace
\sum_{\mumf} \(
- \intUT\iu \phi \iua S(\icu p)\del_t\iua\varphi \dut
+ \intUT \kSua(\icu p) ~\iu\Lambda \nabla \icua u \cdot \nabla \iua \varphi \dut \\
&\eqskip- 
\intU\iu \phi \iua S(\icut[0] p)\iua \varphi(0,\cdot) \du
\)
+ \sum_{\aa\in\chi}\(
\intGT[\aa] \iaaa\rho \jump{\ia\varphi} \dtaut \\
& \eqskip- 
\intGT[\aa] \eta\iaaa S(\tracemf\icm p) \del_t\tracemf\ima\varphi \dtaut
-\intG[\aa] \eta\iaaa S(\tracemf\icmt[0] p) \tracemf\ima\varphi(0,\cdot) \dtau
\) 
\Bigg\rbrace\\
&\eqskip= \sum_{\alpha=1}^2
\sum_{\mumf} \intUT \iua h \iua \varphi \dut.
\end{aligned}
\end{align}
Note that Equation \eqref{proof.main.0} also holds for any smooth $\ia\varphi$, by density of
tensorial functions in smooth functions \cite[Appendix D]{DRO01}.
Recalling the weak formulation \eqref{modeleVar}, proving Theorem \ref{th:main.cv} is now all about showing that
\begin{align}
\label{proof.main.3}
\sum_{\aaa}
\intGT[\aa] \iaaa\rho \jump{\ia\varphi} \dtaut
= 
\sum_{\aaa}
\intGT[\aa]\Faaa(\jump{\ica u})\jump{\ia\varphi}\dtaut.
\end{align}
This is achieved by using Minty's trick.

\paragraph*{Second step:} proof that
\begin{align}
\label{proof.main.4}
\limsup_{l\to\infty}
\sum_{\aaa}
\intGT[\aa]\FDaaal(\recjump[^l]{\idal u})\recjump[^l]{\idal u}\dtaut
\leq
\sum_{\aaa}
\intGT[\aa] \iaaa\rho \jump{\ica u} \dtaut.
\end{align}

Having in mind to employ the energy inequality \eqref{eq:energy.disc} with $T_0=T$,
we first establish, for $\mumf$ and $\aachi$, the following convergences as $l\to\infty$:
\begin{align}
&
\intUT \iua h \recu[^l] \idal u \dut
\longrightarrow 
\intUT \iua h \icua u \dut \,,
\label{proof:main.5a}\\
&
\intU \iu B( \iu S(\recu[_S^l]\idtl[l][0] p))\du
\longrightarrow
\intU \iu B( \iu S(\icut[0] p))\du\,,
\label{proof:main.5b}\\
&
\intG[\aa] \iaa B( \iaa S(\recmf[_S^l]\idtl[l][0] p))\dtau
\longrightarrow
\intG[\aa] \iaa B( \iaa S(\tracemf\icmt[0] p))\dtau.
\label{proof:main.5c}
\end{align}
The convergence \eqref{proof:main.5a} is obvious by Theorem \ref{th:cv}.
From the choice \eqref{GradDisc:IC} of the scheme's initial conditions,
together with the consistency of the interpolation operators
$\Interpolm$ and $\Interpolf$, $\iu S(\recu[_S^l]\idtl[l][0] p) \rightarrow \iu S(\icut[0] p)$ in $\LU$ and 
$\iaa S(\recmf[_S^l]\idtl[l][0] p) \rightarrow \iaa S(\tracemf\icmt[0] p)$ in $\LG[\aa]$, as $l\to\infty$.
Then, \eqref{propB:3} and \cite[Lemma A.1]{DT14} yield \eqref{proof:main.5b} and \eqref{proof:main.5c}.

We further show that
\begin{align}
&
\liminf_{l\to\infty}
\intU \iu B( \iu S(\recu[_S^l]\idtl[l][N^l] p))\du
\geq
\intU \iu B( \iu S(\icu p)(T))\du\,,
\label{proof:main.6a}\\
&
\liminf_{l\to\infty}
\intG[\aa] \iaa B( \iaa S(\recmf[_S^l]\idtl[l][N^l] p))\dtau
\geq
\intG[\aa] \iaa B( \iaa S(\tracemf\icm p)(T))\dtau\,,
\label{proof:main.6b}\\
&
\liminf_{l\to\infty}
\intUT \kSua(\recu[^l]\idl p)\iu\Lambda \recgradu[^l]\idal u \cdot \recgradu[^l]\idal u \dut\nonumber\\
&\eqskip\geq
\intUT \kSua(\icu p)\iu\Lambda \grad\icua u \cdot \grad\icua u \dut.\label{proof:main.6c}
\end{align}
By the uniform-in-time weak $L^2$ convergences of Theorem \ref{th:cv.unifT-weakL2}, 
$\iu S(\recu[_S^l]\idtl[l][N^l] p) \rightharpoonup \iu S(\icu p)(T)$ in $\LU$ and 
$\iaa S(\recmf[_S^l]\idtl[N^l] p) \rightharpoonup \iaa S(\tracemf\icm p)(T)$ in $\LG[\aa]$, as $l\to\infty$.
Note also that, since (by assumption) $\iu S$ and $\iaa S$ are not explicitly space-dependent 
on each open set of the formerly introduced partitions of $\U$ and $\G[\aa]$, respectively, then
$\iu B, \iaa B$ are neither.
On these partitions, the conditions of \cite[Lemma 4.6]{DET16} are fulfilled, 
namely $\iu B, \iaa B$ are convex and l.s.c.
This lemma, which essentially states the $L^2$-weak l.s.c.\ of strongly l.s.c.\ convex functions
on $L^2$, establishes \eqref{proof:main.6a} and \eqref{proof:main.6b}.
To show \eqref{proof:main.6c}, apply the Cauchy-Schwarz inequality to write
\begin{align*}
&
\intUT \kSua(\recu[^l]\idl p)\iu\Lambda \grad\icua u \cdot \recgradu[^l]\idal u \dut
\leq
\( \intUT \kSua(\recu[^l]\idl p)\iu\Lambda \grad\icua u \cdot \grad\icua u \dut \)^\frac{1}{2}\\
&\eqskip
\times
\( \intUT \kSua(\recu[^l]\idl p)\iu\Lambda \recgradu[^l]\idal u \cdot \recgradu[^l]\idal u \dut \)^\frac{1}{2}
\end{align*}
and take the inferior limit as $l\to\infty$, using the strong convergence of $\kSua(\recu[^l]\idl p)$ and weak convergence of $\recgradu[^l]\idal u$ to pass to the limit in
the left-hand side and the first term in the right-hand side.

Let us now come back to the proof of \eqref{proof.main.4}.
Plugging the convergences \eqref{proof:main.5a}--\eqref{proof:main.6c}
into \eqref{eq:energy.disc} with $T_0=T$ yields
\begin{align}
&\limsup_{l\to\infty}
\sum_{\aaa}
\intGT[\aa]\FDaaal(\recjump[^l]{\idal u})\recjump[^l]{\idal u}\dtaut
\nonumber\\&\eqskip
\leq
\sum_{\mu,\alpha}
\(
\intUT \iua h \icua u \dut 
-
\intUT \kSua(\icu p)~\iu\Lambda \grad\icua u \cdot \grad\icua u \dut
\)
\nonumber\\&\eqskip
+
\sum_{\mu}
\(
\intU \iu\phi\iu B( \iu S(\icut[0] p))\du
-
\intU \iu\phi\iu B( \iu S(\icu p)(T))\du
\)
\nonumber\\&\eqskip
+
\sum_{\aa}\(
\intG[\aa] \eta\iaa B( \iaa S(\tracemf\icmt[0] p))\dtau
-
\intG[\aa] \eta\iaa B( \iaa S(\tracemf\icm p)(T))\dtau
\).
\label{proof.main.7}
\end{align}
Recall that $C_0^\infty([0, T ))\otimes [C_\O^\infty\times C_{\G}^\infty]$ is
dense in $(\LUT)_{\mumf}$.
Owing to Appendix \ref{sec:ident.tder},
we infer from \eqref{proof.main.0} that $\iff\phi\del_t\ifa S(\icf p)\in L^2(0,T;{\icfSpace[0]V}^\prime)$,
that $\im \phi\del_t\ima S(\icm p)+\sum_{\aa}\tracemf^*
(\eta\del_t\iaaa S(\tracemf\icm p))\in L^2(0,T;{\icmSpace[0]V}^\prime)$
(where $\tracemf^*$ is the adjoint of $\tracemf$), and that, for any
$\ia\varphi\in\icSpace V$,
\begin{align*}
&
\sum_{\alpha=1}^2\Bigg\lbrace
\sum_{\mumf} \(
\intT\langle\iu \phi \del_t\iua S(\icu p),\iua\varphi\rangle \d t
+ \intUT \kSua(\icu p) ~\iu\Lambda \nabla \icua u \cdot \nabla \iua \varphi \dut \)\nonumber\\
&\eqskip+ \sum_{\aa\in\chi}\(
\intGT[\aa] \iaaa\rho \jump{\ia\varphi} \dtaut
+ 
\intT \langle\eta\del_t\iaaa S(\tracemf\icm p),\tracemf\ima\varphi 
\rangle \d t\) 
\Bigg\rbrace\nonumber\\
&\eqskip= \sum_{\alpha=1}^2
\sum_{\mumf} \intUT \iua h \iua \varphi \du.
\end{align*}
Note that the duality product between $(\icfSpace[0]V)'$ and $\icfSpace[0]V$
is taken respective to the measure $\dg(\x) = \iff d(\x)\dtau(\x)$,
and remember the abuse of notation \eqref{abuse.notation}.
Apply this to $\ia\varphi = (\icma u,\icfa u)$. Recalling that $\iua[2] S=1-\iua[1] S$,
we have $\del_t\iua[2] S(\icu p)=-\del_t\iua[1] S(\icu p)$ and thus
\begin{align}
&
\sum_{\mumf}
\intT\langle\iu \phi \del_t\iua S(\icu p),\icu p\rangle \d t+
\sum_{\aa\in\chi}
\intT \langle\eta\del_t\iaaa S(\tracemf\icm p),\tracemf\icm p 
\rangle
\d t\nonumber\\
&+\sum_{\alpha=1}^2 \Bigg\lbrace\sum_{\mumf}\intUT \kSua(\icu p) ~\iu\Lambda \nabla \icua u \cdot \nabla \icua u \dut + \sum_{\aa\in\chi}
\intGT[\aa] \iaaa\rho \jump{\ica u} \dtaut 
\Bigg\rbrace\nonumber\\
&\eqskip= \sum_{\alpha=1}^2
\sum_{\mumf} \intUT \iua h \icua u \dut.
\label{proof.main.0.1}\end{align}

\cite[Lemma 3.6]{DE15} establishes a temporal integration-by-parts property by
using arguments purely based on the time variable, and that can easily be adapted
to our context, even considering the ``combined'' time derivatives
$\im \phi\del_t\ima S(\icm p)+\sum_{\aa}\tracemf^*(\eta\del_t\iaaa S(\tracemf\icm p))$
and the heterogeneities of the media treated here -- i.e.\ the presence of $\iu\phi$,
see assumptions in Section \ref{subsec:modeleCont}. 
This adaptation yields
$$
\intT \langle \iff \phi\del_t\ifa S(\icu p), \icf p \rangle_{{\icfSpace[0]V}^\prime,\icfSpace[0]V} \d t
= \int_{M_f} \iff\phi\iff B( \iff S(\icf p)(T))\dg - \int_{M_f} \iff\phi\iff B( \iu S(\icf p)(0))\dg
$$
and
\begin{align*}
\intT \langle \im \phi\del_t\ima S(\icm p), \icm p \rangle \d t+{}&
\sum_{\aa\in\chi}\intT \langle \eta\del_t\iaaa S(\tracemf\icm p) , \tracemf\icm p 
\rangle
\d t\\
={}&
\int_{M_m} \im\phi\im B( \im S(\icm p)(T))\d\x - \int_{M_m} \im\phi\im B( \im S(\icm p)(0))\d\x
\\
&+\sum_{\aa\in\chi}\(\intG[\aa] \eta\iaa B( \iaa S(\tracemf\icm p)(T))\dtau
- \intG[\aa] \eta\iaa B( \iaa S(\tracemf\icm p)(0))\dtau\).
\end{align*}
Plugging these relations into \eqref{proof.main.0.1} and using the result
in \eqref{proof.main.7} concludes the proof of \eqref{proof.main.4}.

\paragraph*{Third step:} conclusion.

As in the first step, take
$\ia\varphi = (\ima \varphi,\ifa \varphi) = \sum_{k=1}^{b} \ial[k]\theta\otimes\ial[k]\psi$
and set $(\idatl v)_{n=0,\dots,\idl N} = (\proj[_S^l]\ia\varphi(\idtl t))_{n=0,\dots,\idl N}\in(\idlSpace[l][0]X)^{\idl N+1}$.
Developing the monotonicity property \eqref{proof.main.1} of $\FDaaal$ yields
\begin{align*}
&\sum_\aaa\intGT[\aa]\FDaaal(\recjump{\idal u})\recjump{\idal u}\dtaut
- \sum_\aaa\intGT[\aa]\FDaaal(\recjump{\idal v}) ( \recjump{\idal u}-\recjump{\idal v} )\dtaut\\
&\eqskip - 
\sum_\aaa\intGT[\aa]\FDaaal(\recjump{\idal u})\recjump{\idal v}\dtaut
\geq 0.
\end{align*}
Use \eqref{proof.main.2} and \eqref{weak.conv.F} to pass to the limit in the second and third
integral terms:
\begin{align*}
&\limsup_{l\to\infty}
\sum_{\aaa}
\intGT[\aa]\FDaaal(\recjump{\idal u})\recjump{\idal u}\dtaut \\
&\eqskip\geq
\sum_{\aaa}
\intGT[\aa]\Faaa(\jump{\ia\varphi}) ( \jump{\ica u} - \jump{\ia\varphi} )\dtaut
+ \sum_{\aaa}
\intGT[\aa] \iaaa\rho \jump{\ia\varphi} \dtaut.
\end{align*}
Use \eqref{proof.main.4} and the density of 
the tensorial function spaces 
$C_0^\infty([0, T ))\otimes [C_\O^\infty\times C_{\G}^\infty]$
in $L^2(0,T;\icSpace V)$ (cf. \cite[proposition 2.3]{BHMS16})
to obtain
\begin{align*}
\sum_{\aaa}
\intGT[\aa] \iaaa\rho ( \jump{\ica u} - \jump{\ica v} ) \dtaut
\geq
\sum_{\aaa}
\intGT[\aa]\Faaa(\jump{\ica v}) ( \jump{\ica u} - \jump{\ica v} )\dtaut
\end{align*}
for all $(\ica v)_{\alpha=1,2}\in L^2(0,T;\icSpace V)^2$. The conclusion is now standard
in the Minty trick, see e.g.\ \cite[Proof of Theorem 3.34]{gdm}. For any smooth $(\ia\varphi)_{\alpha=1,2}$, choose $\ica v = \ica u \pm \epsilon\ia\varphi$ and let $\epsilon\to 0$ to derive \eqref{proof.main.3} and conclude the proof.
\qed

\section{Two-phase flow test cases}
\label{sec:num}

We present in this section a series of test cases for two-phase flow through a fractured 2 dimensional reservoir of geometry as shown in figure \ref{reservoir-Test1frac}. The domain $\O$ is of extension $(0,10)\mathrm{m}\times (0,20)\mathrm{m}$ and the fracture width $d_f$ is assumed constant equal to $1\mathrm{\ cm}$. We consider isotropic permeability in the matrix and in the fracture. 
The following geological configuration is considered:
Matrix and fracture permeabilities are $\im\lambda = 0.1$ Darcy and $\iff\lambda = 100$ Darcy, respectively, matrix and fracture porosities are $\im\phi = 0.2$ and $\iff\phi = 0.4$, respectively.

Initially, the reservoir is saturated with water (density $\rho^2 = 1000\ \mathrm{kg}/\mathrm{m}^3$, viscosity $\kappa^2 = 0.001\mathrm{\ Pa.s}$) and oil (density $\rho^1 = 700\ \mathrm{kg}/\mathrm{m}^3$, viscosity $\kappa^1 = 0.005\mathrm{\ Pa.s}$) is injected from below. Also, hydrostatic distribution of pressure is assumed. The oil then rises by gravity, thanks to its lower density compared to water.
At the lower boundary of the domain, we impose constant capillary pressure of $0.1$ bar and water pressure of $3$ bar; at the upper boundary, the capillary pressure is constant equal to $0$ bar and the water pressure is $1$ bar.
Elsewhere, homogeneous Neumann conditions are imposed.

\begin{SCfigure}[][h]
 \centering
  \caption{Geometry of the reservoir under consideration. 
  	DFN in red and matrix domain in blue.
  	$\O = (0,10)$m$\times (0,20)$m and $d_f = 0.01$m.}
  \includegraphics[trim=12cm 0cm 12cm 0cm,clip=true,height=0.25\textheight]{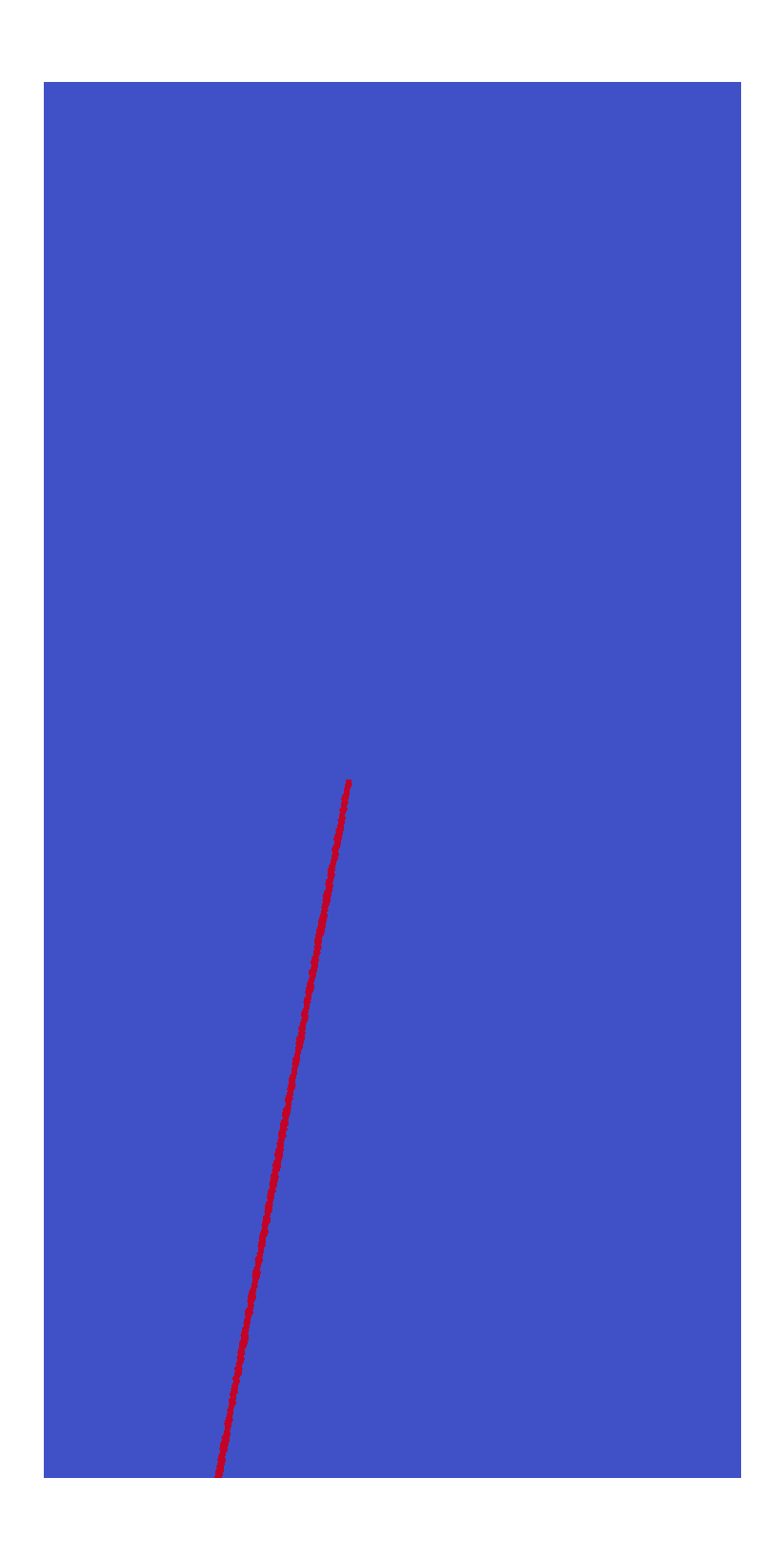}
    \hspace*{1cm}
  \label{reservoir-Test1frac}
\end{SCfigure}

We use the VAG scheme to obtain solutions for the DFM.
We refer to \cite{BHMS16} for a presentation of the scheme as a gradient scheme,
and for proofs that, under standard regularity assumptions on the meshes, the
corresponding sequences of gradient discretisations are coercive, GD-consistent,
limit-conforming and compact. 
The tests are driven on a triangular mesh extended to a 3D mesh with one layer of prisms (we use a 3D implementation of the VAG scheme). 
The resulting numbers of cells and degrees of freedom are exhibited in Table \ref{table_meshes-Test3frac}. 
The mesh size is of order $10 d_f$.

The non-linear system of equations occurring at each time step
is solved via a Newton algorithm with relaxation. 
To solve the linear system obtained at each step of the Newton iteration,
we use the sequential version of the
SuperLU direct sparse solver \cite{superlu_ug99,superlu99}.
The stopping criterion on the $L^1$ relative residual is ${\rm crit}_{\rm Newton}^{\rm rel}$.
To ensure well defined values for the capillary pressure, after each Newton iteration, 
we project the (oil) saturation
on the interval $[0,1-10^{-14}]$.
The time stepping is progressive,
i.e. after each iteration, the upcoming time step is deduced by multiplying the previous one by 2,
while imposing a maximal time step $\Delta t_{max}$.
If at a given time iteration the Newton algorithm does not converge after 35 iterations,
then the actual time step is divided by 4 and the time iteration is repeated.
The number of time step failures at the end of a simulation is 
indicated by $\textbf{N}_{\rm Chop}$.

\begin{table}[H]
\begin{center}
\resizebox{0.7\textheight}{!}{
\begin{tabular}{|c|c|c|c|c|c|}
 \hline
\multicolumn{1}{|l|}{\textbf{Nb Cells}} 
& \multicolumn{1}{|l|}{\textbf{Nb DOF}} & \multicolumn{1}{l|}{\textbf{Nb DOF el.}}
& \multicolumn{1}{|l|}{${\rm crit}_{\rm Newton}^{\rm rel}$} &\multicolumn{1}{|l|}{$\Delta t_{max}$ for $0 \leq t \leq 1/2 ~\mathrm{d}$}
&\multicolumn{1}{|l|}{$\Delta t_{max}$ for $1/2 ~\mathrm{d} < t \leq 10 ~\mathrm{d}$}
\\ \hline
5082 &       10610 &        5528 &		$1.E{-6}$	&		$0.01 ~\mathrm{d}$    &		$0.19 ~\mathrm{d}$ 	  	\\\hline
\end{tabular}}
\caption{\textbf{Nb Cells} is the number of cells of the mesh;
\textbf{Nb DOF} is the number of discrete unknowns; \textbf{Nb DOF el.} is the number of discrete unknowns after elimination of cell unknowns without fill-in. Time steps used in the simulations in days (d)}
\label{table_meshes-Test3frac}
\end{center}
\end{table}

Inside the matrix domain the capillary pressure function is given by Corey's law
$p_m = -a_m\log (1-S_m)$ with $a_m = 1$ bar. Inside the fracture network, we suppose
$p_f = -a_f\log (1-S_f)$ with $a_f = 0.02$ bar. 
The matrix and fracture relative permeabilities of each phase $\alpha$ are given by Corey's laws
$k_{r,m}^\alpha (\ima S) = (\ima S)^2$ and $k_{r,f}^\alpha (\ifa S) = \ifa S$, and the phase mobilities are defined by
$\iua k (\iua S) = \frac{1}{\ia\kappa} k_{r,m}^\alpha (\iua S)$, $\mumf$ (see Figure
\ref{fig_physical-function-curves}). 
The phase saturations at the interfacial layers are defined by the interpolation 
\begin{equation}
  \label{eq_Sa}
\iaaa S = \theta\ima S + (1-\theta)\ifa S,
\end{equation} 
with parameter $\theta\in [0,1]$. The mapping $\iaaa S:[0,+\infty)\to [0,1)$ is a diffeomorphism so the choice
$$
\kSia = \theta \ima k (\ima S) + (1-\theta) \ifa k (\ifa S).
$$ 
is valid, since this function can be written as $\iaaa k(\iaaa S)$ with $\iaaa k(\xi)=
\theta \ima k (\ima S\circ (\iaaa S)^{-1}(\xi)) + (1-\theta) \ifa k (\ifa S\circ (\iaaa S)^{-1}(\xi))$. 
Finally, the interfacial porosity $\iaa\phi$ is set to 0.2 and 
$$
\iaa d = \frac{d_f}{2}\varepsilon,
$$
with parameter $\varepsilon>0$. The parameter $\eta$ is then defined by 
$
\eta = \iaa\phi\iaa d.
$

Let us start with some remarks.
From the capillary pressure functions (cf. figure \ref{fig_physical-function-curves}), it is obvious that for given $p$, the one-sided jump of the oil saturation is negative, i.e.
\begin{align}
\label{eq.satjump}
\im S(p) - \iff S(p)
< 0.
\end{align}

To account for the interfacial zone properly, the mobilities have to be adjusted by choosing the model parameter $\theta$ depending on the rock type characteristics of the layer. Obviously, $\theta = 0$ refers to a fracture rock type and $\theta = 1$ to a matrix rock type.

On the other hand, with larger $\eta$, the volume of the interfacial layers gets augmented
and the interfacial accumulation terms play a more important role.
The availability of the supplementary volume has a direct impact on the phase front speed inside the fracture during its filling: 
\eqref{eq_Sa}--\eqref{eq.satjump} show that the volume of oil in the interfacial layers is strictly decreasing as a function of $\theta$, given a distribution of capillary pressures. This indicates that, from the accumulation point of view, the fracture front speed should grow with growing $\theta$, and this effect should be enhanced by a larger $\eta$.

\begin{figure}[H]
\begin{center}
\includegraphics[page=2,width=0.49\textwidth]{curves.pdf}
\includegraphics[page=1,width=0.49\textwidth]{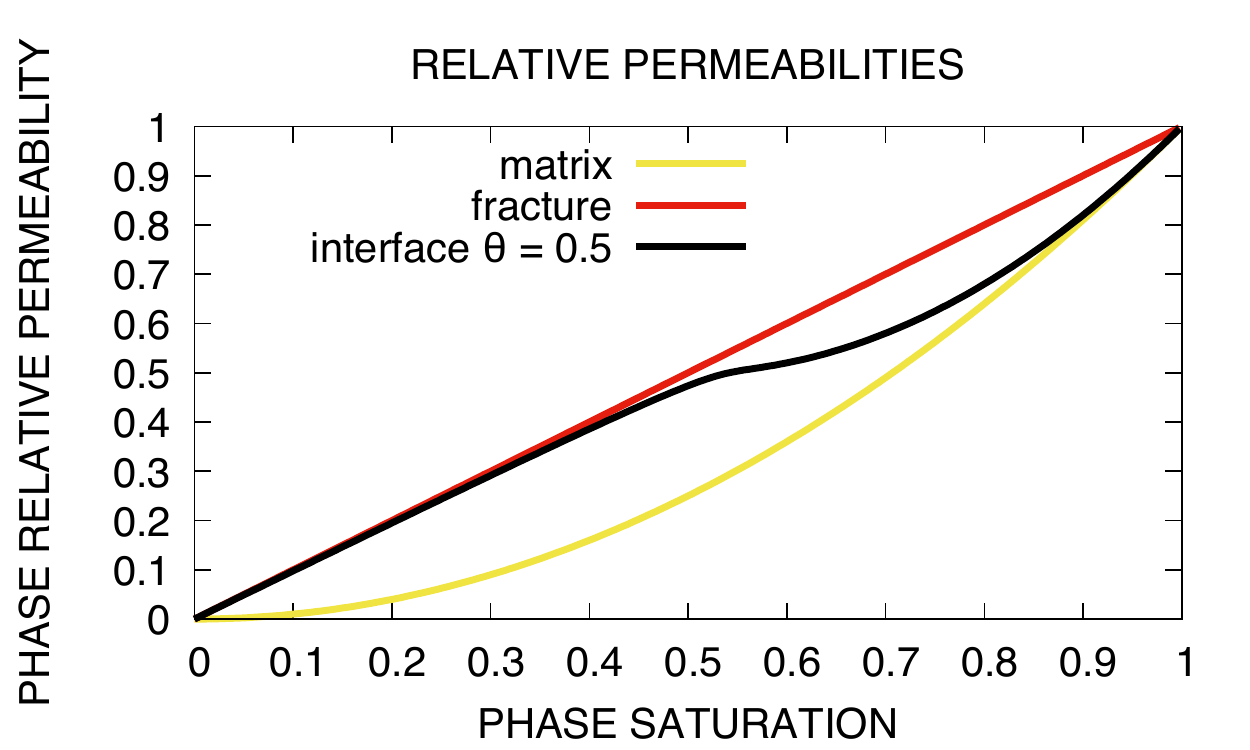}
\end{center}
\caption{
Curves for capillary pressures and relative permeabilities.
}
\label{fig_physical-function-curves}
\end{figure}


\begin{figure}[h!]
\begin{center}
\hspace*{0.5cm}
(a)
\hspace*{0.45\textwidth}
(b)
\\
\includegraphics[page=1,width=0.49\textwidth]{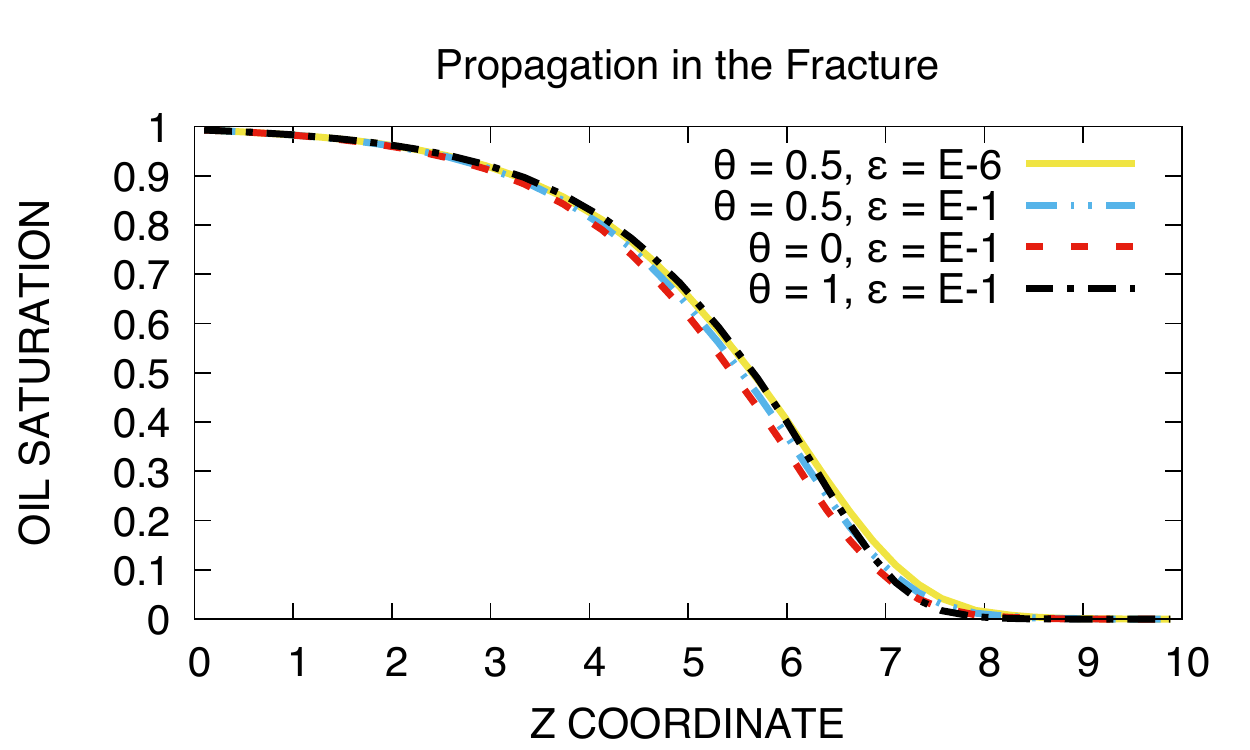}\hspace*{0.1cm}
\includegraphics[page=2,width=0.49\textwidth]{Test1Frac-Frontmf-PUB.pdf}\hspace*{0.1cm}
\\
\hspace*{0.5cm}
(c)
\\
\includegraphics[page=3,width=0.49\textwidth]{Test1Frac-Frontmf-PUB.pdf}\hspace*{0.1cm}
\end{center}
\caption{
Fracture oil saturation for time $t = 6\textrm{h}$.
}
\label{fig_SolFrac}
\end{figure}

Figure \ref{fig_SolFrac} (a) indicates that, for a fixed $\theta = 0, 0.5, 1$, the solutions are not sensitive to small variations of $\varepsilon$.
Quantitatively, we see that the solution for $\varepsilon=0.1$ is close to the solution for $\varepsilon = 10^{-6}$. With respect to the computational performance exposed in Table \ref{tab_computercost}, we thus see that choosing $\varepsilon=0.1$ is a good compromise between accuracy and cost. This point is presented in more detail for the intermediate rock type, i.e. $\theta=0.5$, in Figure \ref{fig_vol1mat}.
Figure \ref{fig_SolFrac} (b) confirms the aforementioned feature of extended (large $\varepsilon$) interfacial layers to delay the propagation of the oil in the drain. As suggested, this effect is even more important, with decreasing $\theta$.
In Figure \ref{fig_SolFrac} (c), we study the impact of the choice of the interfacial mobility for parameters $\theta = 0, 0.5, 1$ on the solution. Here, the interfacial accumulation is negligible due to an $\varepsilon$ close to zero. 
Let us shortly remark that in the limit of a vanishing interfacial layer, i.e. $\eta=0$, we want to recover the fracture mobilities for the mass exchange fluxes between the matrix-fracture interface and the fracture. Hence, in this case, the right choice of $\theta$ would be $0$.
We observe that changing the mobilities does not much influence the solution, due to the fact that fluxes are mostly oriented from the fracture towards the interfacial layers. The regions where a difference is observed in the fracture oil front for the different models are those with a small positive oil saturations. There, the relative permeabilities for $\theta = 0$ and $\theta = 0.5$ are very close and the difference to $\theta = 1$ is at its peak; this explains the behaviour of the fracture front for the three models.

\begin{figure}[h!]
\begin{center}
\includegraphics[page=1,width=0.49\textwidth]{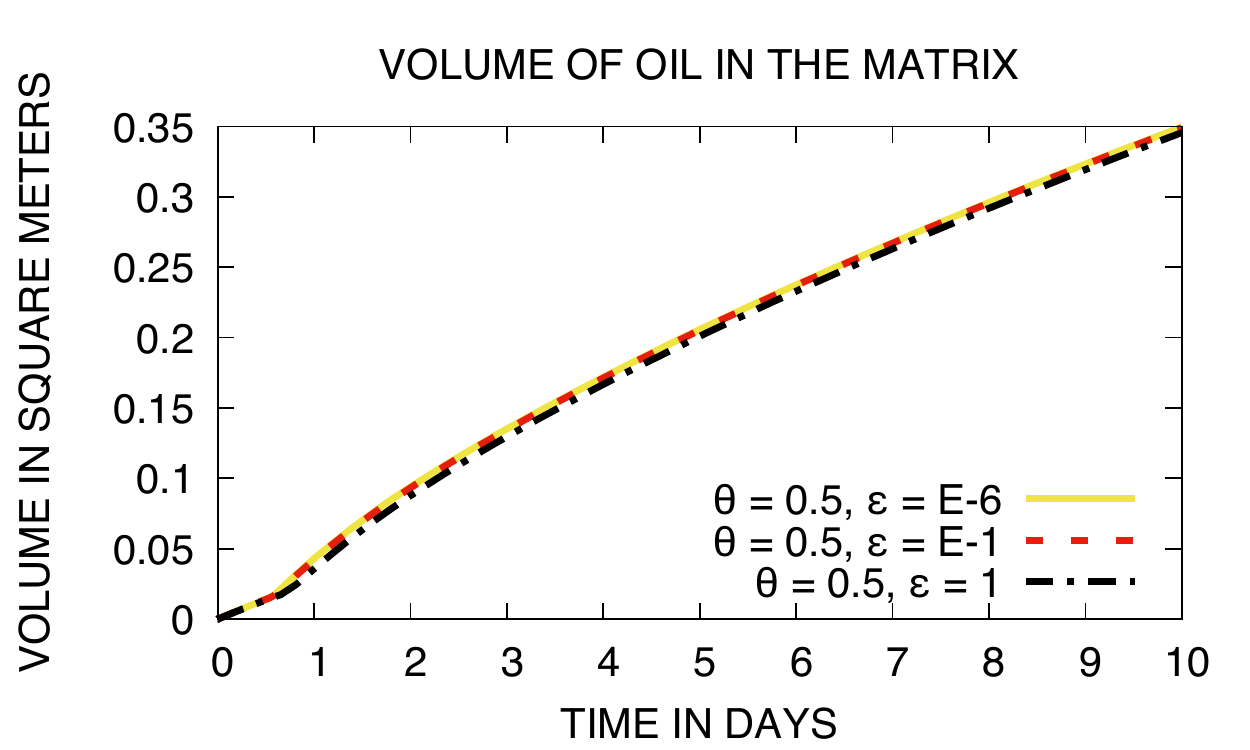}
\includegraphics[page=2,width=0.49\textwidth]{vol1mat-PUB0.pdf}
\\
\includegraphics[page=1,width=0.49\textwidth]{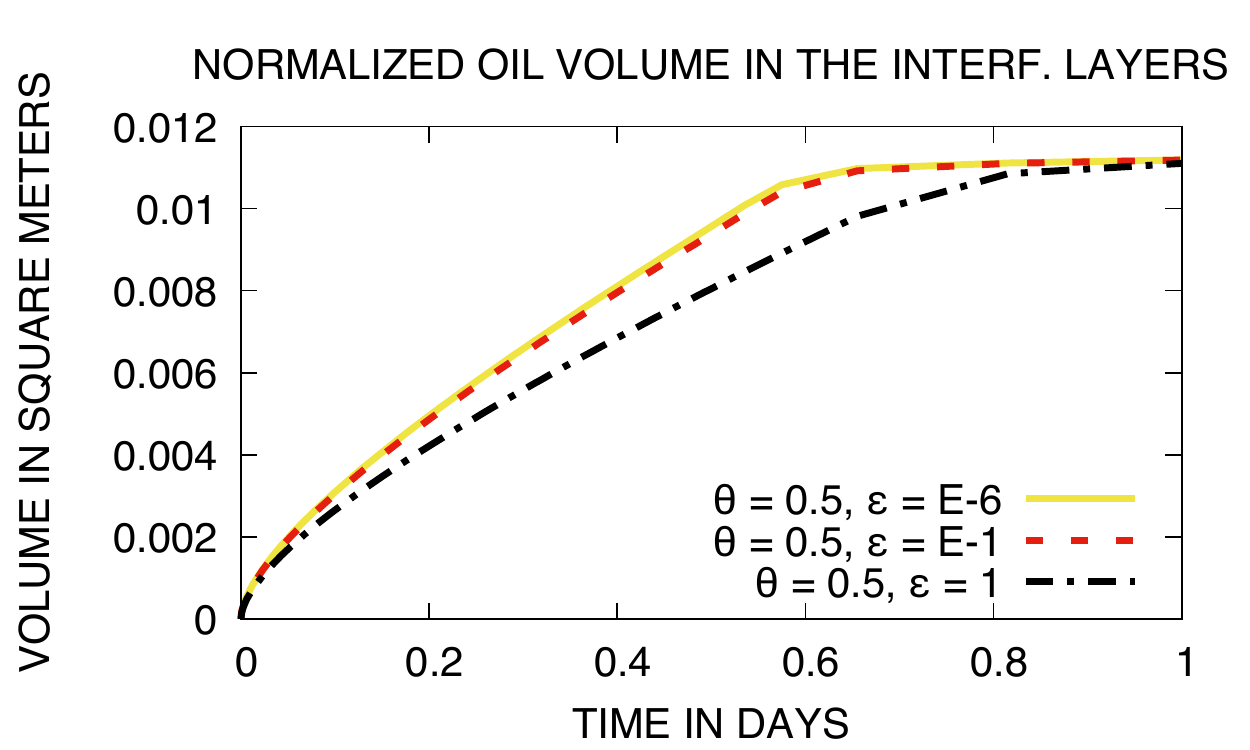}
\end{center}
\caption{
Volume occupied by oil in the matrix, fracture and oil volume normalised by $\varepsilon$ in the interfacial layers,
for $\theta = 0.5$, as a function of time.
}
\label{fig_vol1mat}
\end{figure}

\begin{table}[H]
\begin{center}
\resizebox{0.7\textheight}{!}{
\begin{tabular}{|c|c|c|c|c|c|c|c|c|c|c|c|c|}
\hline
$\theta$ &
\multicolumn{4}{c|}{$0$} &
\multicolumn{4}{c|}{$0.5$} &
\multicolumn{4}{c|}{$1$}
\\\hline
$\varepsilon$ 
&		1		&		1.E-1		&		1.E-6		&		0
&		1		&		1.E-1		&		1.E-6		&		0
&		1		&		1.E-1		&		1.E-6		&		0
\\\hline
$\textbf{N}_{\Delta t}$
&		125		&         125		&         125		&		\multirow{4}{*}{-}
&         125		&         125		&         125		&	\multirow{4}{*}{-}
&         183		&         284		&         377		&	\multirow{4}{*}{-}
\\\cline{1-4}\cline{6-8}\cline{10-12}
$\textbf{N}_{Newton}$
&		506		&         521		&         547		&		{}
&         513		&         521		&         546		&		{}
&         674		&         892		&        1410		&		{}
\\\cline{1-4}\cline{6-8}\cline{10-12}
$\textbf{N}_{\rm Chop}$
&		0		&		0		&        0		&		{}
&           0		&           0		&           0		&		{}
&          22		&          61		&          94		&		{}
\\\cline{1-4}\cline{6-8}\cline{10-12}
\textbf{CPU}
&		147		&		   160		&           159			&		{}
&		   151		&   152		&           170		&		{}
&   402		&   860		&   1402		&		{}\\
\hline
\end{tabular}
}
\end{center}
\caption{Computational cost}
\label{tab_computercost}
\end{table}

Table \ref{tab_computercost} shows that the computational cost increases with decreasing $\varepsilon$ and that, in the case of $\varepsilon=0$, the Jacobian becomes singular. Furthermore, the efficiency severely deteriorates for $\theta = 1$. In this case, $S_\aa'(p)$ is (significantly) smaller during the filling of the fracture (for capillary pressures $p$ below a characteristic $p_1\in\R^+$), since $S_m'(p)\ll S_f'(p)$. When oil fluxes oriented from the fracture to the interface are present, the Jacobian is thus ill-conditioned. 

\section{Conclusion}

We introduced a new discrete fracture model for two phase Darcy flow,
permitting pressure discontinuity at the matrix-fracture interfaces.
It respects the heterogeneities of the media and between the matrix and the fractures,
since it takes into account saturation jumps due to different capillary pressure curves
in the respective domains.
It also considers damaged layers located at the matrix-fracture interfaces.
Another feature of the model are upwind fluxes between these interfacial layers and the fractures.
The upwinding is needed for transport dominated flow in normal direction to the fractures.
The extension to gravity is straightforward (cf. \cite{BHMS16b}).

We developed the numerical analysis of the model in the framework of
the gradient discretisation method,
which contains for example the VAG and HMM schemes.
Based on compactness arguments, we showed in Theorem \ref{th:main.cv} 
the strong $L^2$ convergence of the saturations and the weak $L^2$ convergence
for the pressures to a solution of Model 
\eqref{pde:model}.
In Theorem \ref{th:cv.unifT-weakL2}, we established uniform-in-time, weak $L^2$ in space 
convergence for the saturations, a result that is extended to uniform-in-time, strong $L^2$ in space 
convergence in \cite{DHM.fvca8}.

Finally, we presented a series of test cases, with the objective to study the impact
of the interfacial layer on the solution.
The observed behaviour of the solutions for the different situations
corresponds to the expectations.
It exhibits significant differences, during the filling of the fracture,
for large interfacial layers and small differences for small layers.
In terms of computational cost, we saw that the presence of a damaged zone at the matrix-fracture interface
is needed in order to solve the linear system of the discrete problem, occurring at each time step.
We also observed that for a large contrast between the drain's and the interfacial layer's 
capillary pressures, the simulation becomes expensive. 
Therefore, we see that, in order to cope with both, fractures acting as drains or as barriers,
the possibility to deal with mixed rock types for the damaged zone
is essential.

\appendix
\section{Appendix}

\subsection{Uniform-in-time weak $L^2$ convergence}

Let $A$ be a subset of $\R^n$, endowed with the standard Lebesgue measure,
and $\{\varphi_\ell\,:\,\ell\in\N\}$ be a dense countable set in $L^2(A)$. 
On any bounded ball of $L^2(A)$, the weak topology can be defined by the
following metric:
$$
\dist(v,w) = \sum_{\ell\in\N} \frac{\min{( 1, |\left\langle v - w , \varphi_\ell \right\rangle_{L^2(A)}| )}}{2^\ell}.
$$
A sequence $(v_m)_{m\in\N}$ of bounded functions $[0,T]\to L^2(A)$ converges
uniformly on $[0,T]$ weakly in $L^2(A)$ to some $v$ if it converges uniformly for the weak topology
of $L^2(A)$, meaning that, for all $\phi\in L^2(A)$, $\langle v_m(\cdot),\phi\rangle_{L^2(A)}
\to \langle v(\cdot),\phi\rangle_{L^2(A)}$ uniformly on $[0,T]$ as $m\to\infty$.

With this introductory material, the following result is a consequence of \cite[Theorem 4.26]{gdm} or \cite[Theorem 6.2]{DE15} (see also the reasoning at the end of \cite[Proof of Theorem 3.1]{DE15}).

\begin{theorem}[Discontinuous weak $L^2$ Ascoli--Arzela theorem] \label{th:disc.ascoli}
Let $\mathcal R$ be a dense subset of $L^2(A)$ and
$(v_m)_{m\in\N}$ be a se\-quen\-ce of functions $[0,T]\to L^2(A)$
such that
\begin{itemize}
\item $\sup_{m\in\N}\sup_{t\in [0,T]}\|v_m(t)\|_{L^2(A)}<+\infty$,
\item for all $\varphi\in\mathcal R$,
there exist $\omega_\varphi:[0,T]^2\to [0,\infty)$ and $(\delta_m(\varphi))_{m\in\N}\subset
[0,\infty)$ satisfying
\begin{align*}
&\omega_\varphi(s,t)\to 0\mbox{ as $s-t\to 0$}\,,\;\delta_m(\varphi)\to 0\mbox{ as }m\to\infty\,,\mbox{ and }\\
&\forall (s,t)\in [0,T]^2\,,\;\forall m\in\N\,,\; 
|\langle v_m(s)-v_m(t),\varphi\rangle_{L^2(A)}|\le \delta_m(\varphi)+\omega_\varphi(s,t).
\end{align*}
\end{itemize}
Then, there exists a function $v:[0,T]\to L^2(A)$ such that, up to a subsequence as $m\to\infty$,
$v_m\to v$ uniformly on $[0,T]$ weakly in $L^2(A)$.
Moreover, $v$ is continuous on $[0,T]$ for the weak topology
of $L^2(A)$.
\end{theorem}

\subsection{Generic results on gradient discretisations}

The following lemma is a classical result in the context of the standard gradient discretisation method,
see e.g. \cite[Lemma 4.7]{gdm}. We give a sketch of its proof for gradient discretisations
adapted to discrete fracture model.

\begin{lemma}[Regularity of the Limit]
\label{lemmalimitregularity}
Let $(\D^l)_{l\in \N}$ be a coercive and limit-conforming sequence of gradient discretisations,
and let $( \idl[l]v )_{l\in \N}$ be such that $\idl[l]v\in (\idlSpace[l][0]X)^{N_l+1}$,
where $N_l$ is the number of time steps of $\D^l$. We assume that
$(\|\idl[l]v\|_{\D^l})_{l\in\N}$ is bounded.
Then, there exists $\ic v = (\icm v,\icf v)\in L^2(0,T;\icmSpace[0] V)\times L^2(0,T;\icfSpace[0] V)$
such that, up to a subsequence, the following weak convergences hold:
\begin{equation}\label{reg.lim.conv}
\left\{
\begin{array}{r@{\,\,}c@{\,\,}l}
&\recu[^l] \idl v \rightharpoonup\icu v 
&\quad\mbox{ in } \LUT\,,\mbox{ for $\mumf$},\\
&\recgradu[^l] \idl v \rightharpoonup \nabla \icu v 
&\quad\mbox{ in } \LUT^d\,,\mbox{ for $\mumf$},\\
&\recmf[^l] \idl[l]v \rightharpoonup \tracemf \icm v 
&\quad\mbox{ in } \LGT[\aa],\text{ for all }\aa\in\chi,\\
&\recjump[^l]{\idl[l]v} \rightharpoonup \jump{\ic v}
&\quad\mbox{ in } \LGT[\aa],\text{ for all }\aa\in\chi.
\end{array}
\right.
\end{equation}
\end{lemma}

\begin{proof}
By coercivity and since $(\|\idl[l]v\|_{\D^l})_{l\in\N}$ is bounded, all the sequences
in \eqref{reg.lim.conv} are bounded in their respective spaces. Up to a subsequence, we can
therefore assume that there exists $\icu v\in \LUT$, $\iu{\boldsymbol{\xi}}\in \LUT^d$,
$\iaa\beta\in \LGT[\aa]$ and $\iaa j\in \LGT[\aa]$ such that
$\recu[^l] \idl v \rightharpoonup\icu v$, $\recgradu[^l] \idl v \rightharpoonup \iu{\boldsymbol{\xi}}$, $\recmf[^l] \idl[l]v \rightharpoonup \iaa\beta$ and $\recjump[^l]{\idl[l]v} \rightharpoonup \iaa j$ weakly in their respective $L^2$ spaces as $l\to\infty$.

Take $\q\in\COq\times\CGq$, $\iaa\varphi \in C_0^\infty(\G[\aa])$ and
$\rho\in C^\infty_c(0,T)$. For $F$ a function of $\x$, define
$\rho\otimes F(t,\x)=\rho(t)F(\x)$.
The definition of ${\cal W}_{\D^l_S}$ yields
\begin{align*}
&\Bigg|\int_0^T \intO \(\recgradm[^l] \idl v\cdot (\rho\otimes\im\q) + (\recm[^l] \idl v) \div(\rho\otimes\im \q) \) \d\x \d t\\
&+\int_0^T\intG \( \recgradf[^l] \idl v \cdot (\rho\otimes\iff\q) + (\recf[^l] \idl v) \div_\tau(\rho\otimes\iff\q) \) \dtau(\x)\d t \\
&- 
\sum_{\aa\in\chi}\int_0^T\intG[\aa](\rho\otimes(\im\q\cdot\n_\aa)) \recmf[^l] \idl v \dtau(\x) \d t\\
&+ \sum_{\aa\in\chi}\int_0^T\intG[\aa] (\rho\otimes\iaa\varphi) \( 
\recmf[^l] \idl v - \recf[^l] \idl v - \recjump[^l]{\idl v}\)\dtau(\x)\ dt\Bigg|\\
&\eqskip\le \|\idl v\|_{\D^l}\|\rho\|_{L^2(0,T)}{\cal W}_{\D^l_S}(\q,\iaa\varphi).
\end{align*}
The limit-conformity shows that the right-hand side of this inequality tends to $0$.
Hence,
\begin{align*}
&\int_0^T \intO \(\im{\boldsymbol{\xi}}\cdot (\rho\otimes\im\q) + \icm v \div(\rho\otimes\im \q) \) \d\x \d t\\
&+\int_0^T\intG \( \iff{\boldsymbol{\xi}} \cdot (\rho\otimes\iff\q) + \icf v \div_\tau(\rho\otimes\iff\q) \) \dtau(\x)\d t \\
&- \sum_{\aa\in\chi}\int_0^T\intG[\aa](\rho\otimes(\im\q\cdot\n_\aa)) \iaa\beta \dtau(\x) \d t\\
&+ \sum_{\aa\in\chi}\int_0^T\intG[\aa] (\rho\otimes\iaa\varphi) \( 
 \iaa\beta - \icf v - \iaa j\)\dtau(\x)\ dt=0.
\end{align*}
Applying this to $(\q,\iaa\varphi)=((\im \q,0),0)$ with $\im \q\in C_0^\infty(\OG)^d$, and
using the density of tensorial functions $\{\sum_{r=1}^N \rho_r\otimes \im\q\,:\,N\in\N\,,\;
\rho_r\in C^\infty_c(0,T)\,,\;\im\q\in C_0^\infty(\OG)^d\}$ in
$C^\infty_c((0,T)\times\OG)^d$ (see \cite[Appendix D]{DRO01})
shows that $\im{\boldsymbol{\xi}}=\nabla\icm v$. With $(\q,\iaa\varphi)=((0,\iff \q),0)$
where $\iff\q\in C^\infty(\ov\G_i)^{d-1}$, we obtain $\iff{\boldsymbol{\xi}}=\nabla \icf v$.
Considering now $(\q,\iaa\varphi)=((\im \q,0),0)$ with $\im q\in C_b^\infty(\OG)^d$
and applying the divergence theorem gives $\iaa\beta=\tracemf \icm v$.
Finally, taking $(\q,\iaa\varphi)=((0,0),\iaa\varphi)$ with a general $\iaa\varphi
\in C_0^\infty(\G[\aa])$ yields $\iaa j =\iaa\beta-\icf v=\tracemf \icm v-\icf v
=\jump{\ic v}$. \qed
\end{proof}

With \cite[Lemma 3.6]{EGHM13}, we can state the following.

\begin{corollary}
\label{corollary:limitregularity}
Under the assumptions of Lemma \ref{lemmalimitregularity},
if $\iu g:\R\rightarrow\R\ (\mumf)$ and $\iaa g:\R\rightarrow\R\ (\aachi)$ 
are continuous, non-decreasing functions and if 
$(\recu[^l] \iu g(\idl v))_l$ strongly converges in $\LUT$
and $(\recmf[^l] \iaa g(\idl v))_l$ strongly converges in $\LGT[\aa]$,
then
$$
\left\{
\begin{array}{r@{\,\,}c@{\,\,}l}
&\recu[^l] \iu g(\idl v) \rightarrow \iu g(\icu v)
&\quad\mbox{ in } \LUT,\\
&\recmf[^l] \iaa g(\idl v) \rightarrow \iaa g(\tracemf\icm v)
&\quad\mbox{ in } \LGT[\aa].
\end{array}
\right.
$$
\end{corollary} 

\subsection{Identification of time derivatives}\label{sec:ident.tder}

We discuss here how weak formulations, with derivatives on test functions,
enable us to recover some regularity properties on time derivatives
of quantities of interest.

Let us start with a classical situation, similar to \cite[Remark 1.1]{DE15}.
Let $(M,\nu)$ be a measured space and $E$ be a Banach space densely embedded in $L^2(M)$,
so that $E\hookrightarrow L^2(M)\hookrightarrow E'$. Assume also that $E'$ is separable.
Let $\mathcal L:L^2(0,T;E)\to \R$ be a continuous linear form and
let $\mathcal E\subset C^1_0([0,T);E)$ be such that $\mathcal E_0=\{\Phi\in\mathcal E\,:\,
\Phi(0,\cdot)=0\}$
is dense in $L^2(0,T;E)$.
Suppose that $U\in L^2(0,T;E)$ and $U_0\in L^2(M)$ satisfy, for all $\Phi\in \mathcal E$,
\begin{equation}\label{tder.1}
-\int_0^T \int_M U(t,\x)\del_t\Phi(t,\x)\d\nu(\x)\d t
+ \int_M U_0(\x)\Phi(0,\x)\d\nu(\x)
=\mathcal L(\Phi). 
\end{equation}
This relation shows that 
$$
\Xi:\Phi\mapsto -\int_0^T \int_M U(t,\x)\del_t\Phi(t,\x)\d\nu(\x)\d t
$$
is linear (equal to $\mathcal L$) on $\mathcal E_0$, and continuous for the topology of $L^2(0,T;E)$.
By density of $\mathcal E_0$ in this space, $\Xi$ can be extended into an element
of $(L^2(0,T;E))'=L^2(0,T;E')$ (see \cite[Theorem 1.4.1]{DRO01}). We denote this element
by $\del_t U$, as it clearly corresponds to the distributional derivative
of $U$ \cite[Section 2.1.2]{DRO01}. By \cite[Section 2.5.2]{DRO01}
this shows that $U\colon [0,T]\to L^2(M)$ is continuous and, using \cite[Proposition 2.5.2]{DRO01} to
integrate by parts in \eqref{tder.1}, that $U(0)=U_0$ and
\begin{equation}\label{tder.2}
\forall \Phi\in\mathcal E\,,\;
\langle \del_t U,\Phi\rangle_{L^2(0,T;E'),L^2(0,T;E)}\d t=
\int_0^T \langle \del_t U(t),\Phi(t)\rangle_{E',E}\d t
=\mathcal L(\Phi). 
\end{equation}
By density of $\mathcal E$ in $L^2(0,T;E)$, this relation actually holds for any
$\Phi\in L^2(0,T;E)$.

Fixing $M=M_f$, $\d\nu=\dg$, $E=V_f^0$, $\mathcal E=C^1([0,T];C^\infty_\Gamma)$ and
\begin{align*}
\mathcal L(\Phi)={}&
 \intT\int_{M_f} \ifa h \Phi \dg\d t
- \intT\int_{M_f} \kSfa(\icf p) ~\iff\Lambda \nabla \icfa u \cdot \nabla \Phi \dg\d t \\
&+ \sum_{\aa\in\chi}\(\intGT[\aa] \iaaa\rho (-\Phi) \dtaut \),
\end{align*}
and using \eqref{proof.main.0} with $\ima \varphi=0$ and 
$\ifa\varphi=\Phi,~\ifa[\beta]\varphi=0$, for $\alpha,\beta=1,2$
with $\alpha\neq\beta$,
this identifies $\del_t(\iff\phi\ifa S(\icf p))=\iff\phi\del_t\ifa S(\icf p)$ as an element of $L^2(0,T;{\icfSpace[0]V}^\prime)$.

\medskip

Let us now consider a slightly more complicated case, in which the time derivatives
of two functions need to be combined to exhibit a certain regularity.
With the same $M$ and $E$ as above, take $(N,\lambda)$ a measured space and 
$\gamma:E\to L^2(N)$ a continuous linear
mapping. Assume that $U\in L^2(0,T;E)$, $V\in L^2(0,T;L^2(N))$,
$U_0\in L^2(M)$ and $V_0\in L^2(N)$, satisfy, for all $\phi\in \mathcal E$,
\begin{equation}
\begin{aligned}\label{tder.3}
-\int_0^T \int_M U(t,\x)\del_t{}&\Phi(t,\x)\d\nu(\x)\d t
-\int_0^T \int_N V(t,\x)\del_t\gamma(\Phi(t))(\x)\d\lambda(\x)\d t\\
{}&+ \int_M U_0(\x)\Phi(0,\x)\d\nu(\x)+ \int_N V_0(\x)\gamma(\Phi(0))(\x)\d\lambda(\x)
=\mathcal L(\Phi). 
\end{aligned}
\end{equation}
The same reasoning as above shows that 
$$
\widetilde{\Xi}:\Phi\mapsto -\int_0^T \int_M U(t,\x)\del_t\Phi(t,\x)\d\nu(\x)\d t
-\int_0^T \int_M V(t,\x)\del_t\gamma(\Phi(t))(\x)\d\lambda(\x)\d t
$$
can be extended into a linear continuous form on $L^2(0,T;E)$.
Letting $\gamma^*: L^2(N)\to E'$ be the adjoint of $\gamma$ (that is,
$\langle g,\gamma(\Phi)\rangle_{L^2(N)}=\langle \gamma^*g,\Phi\rangle_{E',E}$
for all $g\in L^2(N)$ and $\Phi\in E$), the form $\widetilde{\Xi}$ is naturally denoted by $\del_t U + \gamma^* \del_t V$.
Note that, in this sum, the two terms cannot be separated and it cannot,
for example, be asserted that $\del_t U\in L^2(0,T;E')$ and $\gamma^*\del_t V\in L^2(0,T;E')$.
Then, a reasoning similar to the one in \cite{DRO01} shows that
$U+\gamma^* V\colon [0,T]\to L^2(M)$ is continuous with value $U_0+\gamma^*V_0$ at $t=0$,
and that, for all $\Phi\in L^2(0,T;E)$,
$$
\langle \del_t U+\gamma^*\del_t V,\Phi\rangle_{L^2(0,T;E'),L^2(0,T;E)}
=\mathcal L(\Phi). 
$$
To write more natural equations, in the rest of the paper we sometimes make an abuse
of notation and separate the two derivatives. We then write
\begin{equation}
\label{abuse.notation}
\begin{aligned}
\langle \del_t U+\gamma^*\del_t V,\Phi\rangle_{L^2(0,T;E'),L^2(0,T;E)}
={}&\int_0^T\langle \del_t U,\Phi\rangle\d t+
\int_0^T\langle \gamma^*\del_t V,\Phi\rangle\d t\\
={}&\int_0^T\langle \del_t U,\Phi\rangle\d t+
\int_0^T\langle \del_t V,\gamma\Phi\rangle\d t,
\end{aligned}
\end{equation}
where, in the right-hand side, the duality brackets do not have indices,
to avoid claiming that $\del_t U\in L^2(0,T;E')$ or $\gamma^*\del_t V\in L^2(0,T;E')$,
and to remember that these two terms must be understood together.

Used in \eqref{proof.main.0} with $\gamma=\gamma_\aa$ for all $\aa\in\chi$, the
above reasoning and notations enable us
to identify the (combined) time derivatives of 
$\im \phi \ima S(\icm p)$ and $\sum_{\aa}\eta\iaaa S(\tracemf\icm p)$
as elements of $L^2(0,T;{\icmSpace[0]V}^\prime)$.


\section*{Acknowledgment}
The authors would like to thank Total S.A. 
and the Australian Research Council's Discovery Projects funding scheme
(project number DP170100605) for supporting this work.

\bibliographystyle{abbrv}
\bibliography{GradScheme2Phase}

\end{document}